\documentclass[a4paper,10pt]{amsart}

\usepackage{etoolbox}
\usepackage{setspace} %linespacing
\usepackage{amsthm} % for proof
\allowdisplaybreaks
\usepackage{stmaryrd} % for \llbracket
\usepackage{latexsym}
\usepackage{amsmath}
\usepackage{amssymb}
\usepackage{amscd}
\usepackage{mathrsfs} % for \mathscr
\usepackage{xcolor}
\usepackage[pagebackref,breaklinks,linktoc=all,hyperfootnotes=false]{hyperref} % Color of links and page back ref
\hypersetup{
	colorlinks  =	true,
	linkcolor	=	{green!50!black},
	citecolor	=	{green!50!black},
	urlcolor    =   {green!50!black}
}
\usepackage{mathrsfs}
\usepackage[utf8x]{inputenc}
\usepackage{enumitem} % For enumerate with letters
\usepackage{ucs}
\usepackage{graphicx}
\usepackage{tabularx}
\usepackage{wrapfig}
\usepackage{pgfplots}
\usepackage{caption}
\usepackage{subfigure}
\usepackage{longtable}
\usepackage{multicol}
\usepackage{mathtools}
\usepackage{tikz-cd}
\usepackage{tikz}

\usepackage[margin=1in]{geometry}
\usepackage[skins]{tcolorbox}
\usepackage{minibox}

\usepackage{float} 	% It forces tables in position with "[H]".

\usepackage[ruled,vlined]{algorithm2e} %For writing algorithms

\makeatletter 
\def\@tocline#1#2#3#4#5#6#7{\relax
	\ifnum #1>\c@tocdepth % then omit
	\else
	\par \addpenalty\@secpenalty\addvspace{#2}%
	\begingroup \hyphenpenalty\@M
	\@ifempty{#4}{%
		\@tempdima\csname r@tocindent\number#1\endcsname\relax
	}{%
		\@tempdima#4\relax
	}%
	\parindent\z@ \leftskip#3\relax \advance\leftskip\@tempdima\relax
	\rightskip\@pnumwidth plus4em \parfillskip-\@pnumwidth
	#5\leavevmode\hskip-\@tempdima
	\ifcase #1
	\or\or \hskip 1em \or \hskip 2em \else \hskip 3em \fi%
	#6\nobreak\relax
	\dotfill\hbox to\@pnumwidth{\@tocpagenum{#7}}\par% <---- \dotfill -> \hfill
	\nobreak
	\endgroup
	\fi}
\makeatother

\date{\today}
\author{Luca Dall'Ava}
\title{Approximations of the balanced triple product $p$-adic $L$-function}

\newcommand{\z}{\mathbb{Z}}

\newcommand{\zp}{\mathbb{Z}_p}
\newcommand{\q}{\mathbb{Q}}
\newcommand{\cp}{\mathbb{C}_p}
\newcommand{\f}{\mathbb{F}}

\newcommand{\qp}{\mathbb{Q}_p}

\newcommand{\A}{\mathbb{A}}

\newcommand{\cc}{\mathbb{C}}

\newcommand{\ph}[1]{\phantom{#1}}

\newcommand{\Addresses}{
	{
		\bigskip
		\footnotesize
		
		L. Dall'Ava, \textsc{Dipartimento di Matematica, Università degli Studi di Milano,
			Milano, Italy.}\par\nopagebreak
		\textit{E-mail address:} \href{mailto:luca.dallava@unimi.it}{\texttt{luca.dallava@unimi.it}}\par\nopagebreak
		\textit{URL:} \url{https://sites.google.com/view/luca-dallava/}
		
	}
}

\newtheorem{theoremaleph}{Theorem}

\newtheorem{theorem}{Theorem}[subsection]
\newtheorem{definition}[theorem]{Definition}
\newtheorem{lemma}[theorem]{Lemma} 

\newtheorem{proposition}[theorem]{Proposition}
\newtheorem{assumption}[theorem]{Assumption}
\newtheorem{remark}[theorem]{Remark}
\newtheorem{example}[theorem]{Example}
\newtheorem{routine}[theorem]{Algorithmic routine}
\newtheorem{notation}[theorem]{Notation}
\newtheorem*{question*}{Question}

\numberwithin{equation}{subsection}

\newcommand{\commenttt}[1]{}

\newcommand{\mat}[4]{\left(\begin{smallmatrix}
		#1 & #2\\
		#3 & #4
	\end{smallmatrix}\right)}
\newcommand{\Mat}[4]{\left(\begin{matrix}
		#1 & #2\\
		#3 & #4
	\end{matrix}\right)}

\renewcommand{\gcd}{\mathrm{gcd}}

\newcommand{\lcm}{\mathrm{lcm}}

\newcommand{\hla}[1]{\hypertarget{#1}{#1}}
\newcommand{\hrf}[1]{\hyperlink{#1}{#1}}

\pgfplotsset{compat=1.16}

\begin{document}
	
	\begin{abstract}
		The main purpose of this note is to provide an algorithm for approximating the value of the balanced $p$-adic $L$-function, as constructed in \cite{Hsieh2021}, at the point $(2,1,1)$, which is lying outside of the interpolation region. The algorithmic procedure is obtained building on the work of \cite{FrancMasdeu2014} and considering finite-length geodesics on the Bruhat--Tits tree for $GL_2(\qp)$. We are interested in the case where at least one of the Hida families is associated with an elliptic curve over the rationals and we further restrict ourselves to the case where only one finite local sign of the functional equation is $-1$.\newline
		
	\end{abstract}
	
	\maketitle
	
	{   \hypersetup{hidelinks}
		\tableofcontents}
	
	\section{Introduction}
	
	Let $p\geq 5$ and $\ell\neq2$ be two distinct primes. Take $B$ to be the, unique up to isomorphism, definite rational quaternion algebra ramified exactly at $\ell$ and $\infty$. We consider three $p$-adic Hida families, $f_\infty$, $g_\infty$ and $h_\infty$, as follows:
	\begin{enumerate}[label=(\alph*)]
		\item $f_\infty$ is the unique Hida family associated with $f\in S_2(\Gamma_0(N_1\ell p),\q)$, a twist-minimal primitive newform corresponding to a $p$-ordinary elliptic curve $E/\q$. In particular, the family has tame level $N_1$ with trivial tame character;
		\item ${g_\infty}$ and ${h_\infty}$ are $p$-adic Hida families of tame level $N_2\ell$ with tame character $\psi$ and $\psi^{-1}$ respectively. We moreover suppose that $\psi$ and $\psi^{-1}$ are both primitive of square-free conductor $N_2$, with $N_2\mid N_1$. 
	\end{enumerate}
	Under some technical assumptions, which we recall in Section \ref{hypotheses}, Hsieh constructs in \cite{Hsieh2021} the \emph{balanced $p$-adic triple product $L$-function}, $\mathcal{L}^{bal}_{F_\infty}$, associated with the triple $F_\infty=(f_\infty,g_\infty,h_\infty)$. This $p$-adic $L$-function arises on the so-called balanced region, namely the region of triples of arithmetic points $((k_1,\varepsilon_1),(k_2,\varepsilon_2),(k_3,\varepsilon_3))$, such that 
	\begin{equation*}
		k_1+k_2+k_3 \equiv 0\pmod{2} \phantom{000}\textrm{and}\phantom{000} k_1+k_2+k_3>2k_i \phantom{00}\forall i=1,2,3,
	\end{equation*}
	and it extends to the whole product of the weight spaces. For the sake of simplicity, throughout the paper we will recall only the major details about $\mathcal{L}^{bal}_{F_\infty}$ and point the interested reader directly to \cite{Hsieh2021}, for the general construction and properties of the $p$-adic $L$-function. In the balanced region, $\mathcal{L}^{bal}_{F_\infty}$ satisfies the interpolation problem relating its square with the complex triple product $L$-function. On the contrary, outside this region the same relation is no more ensured; one of this points is $(2,1,1)$, where here $2$ and $1$ are points respectively associated with weight-2 and weight-1 specializations. The main motivation behind the study of the point $(2,1,1)$ comes from the renowned \emph{Birch and Swinnerton--Dyer conjecture} and its $p$-adic analogue, as one expects to recover information on the ($p$-adic) $L$-function associated with the weight-2 specialization $f_2=f$. More precisely, this analysis originates from a conjecture of Bertolini--Seveso--Venerucci and from the desire to provide algorithmic support to it. We refer the interested reader to Section \ref{meaning  and motivation} for more details.
	
	In this note, we deal with the above situation from a computational point of view and we provide a routine for approximating the value of $\mathcal{L}^{bal}_{F_\infty}$ at $(2,1,1)$. The $p$-adic $L$-function is constructed as the limit of certain theta-elements defined and studied in Section 4.6 of \cite{Hsieh2021}, in particular in Proposition 4.9; we describe explicitly these theta-elements and provide an algorithm for computing their values when evaluated at a triple of arithmetic points of the form $(2,(2,\varepsilon),(2,\varepsilon))$, for $\varepsilon$ a primitive $p$-adic character of conductor $p^n$. This allows to approximate the value $\mathcal{L}^{bal}_{F_\infty}(2,1,1)$ as the limit over the increasing conductor $p^n$ of such theta-elements evaluated at $(2,(2,\varepsilon),(2,\varepsilon))$. We remark here that the restriction to a definite quaternion algebra ramified at only one rational prime is not strictly necessary and the approach used can be generalized to deal with more primes of ramification. On the other hand, one necessarily requires $f_\infty$ to have trivial tame character and $g_\infty$ and $h_\infty$ to have opposite tame character; this last hypothesis is mandatory to isolate the component associated with $f$ and to express the one depending on $g_\infty$ and $h_\infty$ as sum of functions of the finite-length geodesics on the Bruhat--Tits tree.
	
	Our strategy relies on the fact that, considering weight-$2$ specializations together with the Approximation Theorem (see Lemma \ref{p-adic model}), one can interpret quaternionic modular forms, with trivial character, as functions on suitable $p$-adic double quotients. Section \ref{geodesics} is devoted to study part of these $p$-adic double quotients, namely $GL_2(\qp)/\qp^\times\Gamma_0(p^n\zp)$, which can be identified with the space of length-$n$ geodesics on the Bruhat--Tits tree for $GL_2(\qp)$. Moreover, the canonical projection maps correspond to forgetting the end point of a geodesic, as observed in Lemma \ref{proj maps}. In Section \ref{classes of representatives} and Section \ref{the algorithm} we supply explicit matrix representatives for $GL_2(\qp)/\qp^\times\Gamma_0(p^n\zp)$ and extend the algorithm developed in \cite{FrancMasdeu2014} to the case of finite-length geodesics (see Algorithm \ref{algorithm}). Consider now the three specializations $(f,g_{(2,\varepsilon)},h_{(2,\varepsilon)})$ and the triple of eigenvalues at $p$,
	$(a_p(f), a_p(g_{(2,\varepsilon)}\otimes \varepsilon_\A^{-1}), a_p(h_{(2,\varepsilon)}))$. Let $F^1$ be the quaternionic modular form associated with $f$ (which sometimes corresponds to the harmonic cocycle associated with $f$). Taking $G^n$ (resp.  $H^n$) to be the quaternionic modular form (for quaternionic families chosen as in \cite{Hsieh2021}, Theorem 4.5) associated with $g_{(2,\varepsilon)}\otimes \varepsilon_\A^{-1}$ (resp. $h_{(2,\varepsilon)}$) we can identify the product $G^n \cdot H^n$ as a function of the length-$n$ geodesics on the Bruhat--Tits tree. We hence obtain the following theorem (see Proposition \ref{Theta element} and Theorem \ref{theorem approximation L}). 
	\begin{theoremaleph}
		The theta-element appearing in the triple product $p$-adic $L$-function $\mathcal{L}^{bal}_{F_\infty}$, evaluated at $(2,(2,\varepsilon),(2,\varepsilon))$, is equal to
		\begin{multline}\tag{1}\label{eq thm b}
			\Theta_{{F_\infty^{B}}'}\left(((2,\varepsilon),2,(2,\varepsilon))\right)=\frac{(1-p^{-1})}{a_p(f)^n  }\cdot \sum_{e\in \Gamma\backslash \mathcal{E}(\mathcal{T})} \frac{F^1(e)}{\# Stab_{\left(R^1\left[1/p\right]\right)^1}(e)}\cdot\\ 
			\cdot \left(\frac{\varepsilon^{-1}_{\A}\left(p^n\right)\cdot p^{2n}}{a_p\left(g_{(2,\varepsilon)}\otimes \varepsilon_\A^{-1}\right)^n a_p\left(h_{(2,\varepsilon)}\right)^n}\right)
			\sum_{g \in Geod_n(\mathcal{T})(e)} G^n\left(g\left(\begin{smallmatrix}
				1 & p^{-n}\\
				0 & 1
			\end{smallmatrix}\right)\right)\cdot H^n\left(g\left(\begin{smallmatrix}
				0 & p^{-n}\\
				-p^{n} & 0
			\end{smallmatrix}\right)\right).
		\end{multline}
		Furthermore, by construction, the value at $(2,1,1)$ of the balanced $p$-adic $L$-function (up to a unit in the fraction field of the Iwasawa algebra associated with the triple of Hida families) is
		\begin{align*}
			\mathcal{L}^{bal}_{F_\infty}(2,1,1)\overset{\cdot}{\equiv} \lim\limits_{\varepsilon\rightarrow 1}  \Theta_{{F_\infty^{B}}'}((2,\varepsilon),2,(2,\varepsilon))
		\end{align*}
		and the limit $\lim\limits_{\varepsilon\rightarrow 1} \Theta_{{F_\infty^{B}}'}((2,\varepsilon),2,(2,\varepsilon))$ can be algorithmically approximated with a given $p$-adic precision.
	\end{theoremaleph}
	\noindent The outer sum in Equation (\ref{eq thm b}) is independent of the limit process; it depends only on $f$ and the representatives of the edges $\mathcal{E}(\mathcal{T})$ of the Bruhat--Tits tree, modulo the action of $\Gamma$, the image in $PGL_2(\qp)$ of the invertible elements of a suitable Eichler $\z[1/p]$-order $R^1\left[1/p\right]$. Differently, the inner component (together with $a_p(f)^n$) varies during the limit process and each summation depends on the class of an edge and on the values of $G^n$ and $H^n$ on length-$n$ geodesics.\newline
	
	\textbf{Acknowledgments:} This note represents one of the results in the author's doctoral dissertation \cite{DallAva2021PhD}. He wishes to express his gratitude to his supervisor, Massimo Bertolini, for his encouragement in studying the limit of the triple product $L$-function from both a theoretical and a computational point of view. The author is grateful to the anonymous referee for the valuable comments and suggestions, which helped considerably the improvement of the manuscript. Many thanks go to Marc Masdeu for answering disparate questions about \cite{FrancMasdeuSagecode2011} and the state of the art of computations of Brandt matrices. In the end, the author thanks Jonas Franzel and Matteo Tamiozzo for several helpful discussions. The author is grateful to the Universit\"at Duisburg--Essen, where the main part of this work has been carried out, as well as to the Università degli Studi di Padova (Research Grant funded by PRIN 2017 ``Geometric, algebraic and analytic methods in arithmetic'') and Università degli Studi di Milano (Research fellowship ``Motivi, regolatori e geometria aritmetica'') for their financial support.
	
	\section{Some recalls}\label{Some recalls}
	
	\subsection{Recalls on points on the weight space}\label{weight space}
	
	In this section we recall the main properties of the ($\cp$-points of the) weight space and produce sequences of points on it. We refer the reader to Section 1.4 of \cite{coleman_mazur_1998} for a thorough discussion. We are interested in the $\cp$-points of the weight space, which are identified with $\mathcal{X}_{\zp^\times}\left(\cp\right)=\mathrm{Hom}_{grp}^{cts}\left(\zp^\times,\cp^\times\right)$. 
	% 	The usual decomposition $\zp^\times=\mu_{p-1}\left(\qp\right)\times \left(1+p\zp\right)$ yields the decomposition of the ($\cp$-points of the) weight space into $p-1$ disjoint copies of the unit disk: 
	% 	\begin{equation*}
		% 		\mathcal{X}_{\zp^\times}\left(\cp\right)\cong 		\coprod_{i\in\f_p^\times}\mathrm{Hom}_{grp}^{cts}\left(1+p\zp,\cp^\times\right).
		% 	\end{equation*}
	We fix the dense embedding of $\z$ in $\mathcal{X}_{\zp^\times}\left(\zp\right)=\mathrm{Hom}_{grp}^{cts}\left(\zp^\times,\cp^\times\right)$ via the association
	\begin{equation*}
		\z\ni k \longmapsto \left[x\mapsto x^{k-2}\right]\in \mathcal{X}_{\zp^\times}\left(\zp\right),
	\end{equation*}
	and we refer to the points in $\z\cap \mathcal{X}_{\zp^\times}$ as the \emph{integer weights} or the \emph{integer points} of the weight space. Let $\Lambda=\zp\llbracket \zp^\times\rrbracket$ be the Iwasawa algebra. We consider the following characters:
	\begin{enumerate}
		\item The (induced) Teichm\"uller character $\omega:\Lambda \longrightarrow \cp$ such that $\omega \left(1+p\mathbb{Z} _{p}\right) =1$ and $\omega_{|\left( \mathbb{Z} /p\mathbb{Z} \right) ^{\times}}$ is the usual Teichm\"uller character.
		
		\item The character $\eta _{k}:\Lambda \longrightarrow \cp$ for $k\in\zp$, such that $a\longmapsto \langle\langle a \rangle\rangle^{k}$. Here $\langle\langle- \rangle\rangle$ is the character induced by the projection $\zp^\times\ni a\longmapsto \langle\langle a\rangle\rangle=a\omega(a)^{-1}\in 1+p\zp$.
	\end{enumerate}
	The fixed dense embedding can be written in terms of the above characters as $\z\ni k = \omega^{k-2}\cdot \eta_{k-2}.$
	We define an \emph{arithmetic point} on the weight space to be a point of the form $\chi(-) \cdot k:\zp^\times\longrightarrow \mathcal{O}^\times$, where $\chi$ is a finite order character of $\zp^\times$ and $k$ is an integer point corresponding to $k\in \z_{\geq 2}$. We identify these points with the corresponding couple $(k,\chi)=\omega^{k-2}\cdot \eta_{k-2}\cdot \chi$. In the end, we say that a point is \emph{classical}, with respect to a certain Hida family $f_\infty$, if the specialization of $f_\infty$ at that point is a classical modular form. We recall Hida's (and more generally Coleman's) Classicality Theorem, which ensures that arithmetic points are classical, independently of the Hida family. The inclusion is strict, \emph{e.g.} due to classical forms of weight 1 in the ordinary case.
	
	Aiming to consider families of classical arithmetic points converging to a triple of points with weights $(2,1,1)$, we should study limits to the weight-1 points in the weight space. Such points correspond, under our fixed convention, to elements of the form $\varepsilon\cdot\omega^{-1}\cdot \eta_{-1}$, for $\varepsilon$ a finite character of order a power of $p$.
	
	\begin{remark}\label{Density arithm}
		The weight-2 arithmetic points (and hence the arithmetic points in general) are dense in the weight space. In particular, there exists a sequence of weight-2 arithmetic points converging to $(1,\chi)$, for $\chi$ any finite character of order a power of $p$.
	\end{remark}
	\noindent We point out that, taken $\left\{(2,\varepsilon)=\varepsilon \right\}$ a sequence of weight-2 points converging to a weight-1 point, the sequence of the conductors $cond(\varepsilon_n)=p^{c(\varepsilon_n)}$ has to tend to infinity, for $n$ increasing.	
	% 			\begin{proof}
		% 			    For the sake of contradiction, assume that $c(\varepsilon_n)$ is lesser equal than a certain $M>0$, for all $n$. By continuity of the product of functions, we can rise to the power $p^{M}$ every element of this sequence and note that this becomes the trivial sequence converging to $(1,\psi)^{p^{M}}\neq 1$. This yields a contradiction as that character has infinite order. We obtained that the sequence of the conductors $p^{c(\varepsilon_n)}$ cannot be bounded above. In particular, we can suppose (up to consider a suitable subsequence) that the sequence of the $c(\varepsilon_n)$ is strictly monotone, hence $c(\varepsilon_n)$ tends to infinity as $n$ increases.
		% 			\end{proof}

	%%%%%%%%%%%%%%%%%%%%%%%%%%%%%%%%%%%%%%%%%%%%%%%%%%%%%%%%%%%%%%
	
	\subsection{Recalls on the geodesics on the Bruhat--Tits tree}\label{geodesics}
	
	We collect here some known facts about the Bruhat--Tits tree and we focus our attention on the study of finite-length geodesics. We refer to \cite{Rhodes2001}, the book \cite{Serre1980}, and to \cite{dasgupta_teitelbaum_2008}, for all the proofs and details.
	
	Let $\mathcal{T}$ be the Bruhat--Tits tree for the group $PGL_2(\qp)$. It is the infinite regular tree of valence $p+1$ whose vertices are associated with the quotient space
	\begin{equation*}
		GL_2(\qp)/\qp^\times GL_2(\zp)=PGL_2(\qp)/PGL_2(\zp).
	\end{equation*}
	We set $\mathcal{V}(\mathcal{T})$ to be set of vertices of $\mathcal{T}$ and $\mathcal{E}(\mathcal{T})$ the set of directed edges in $\mathcal{T}$. A directed edge $e$ is represented by an ordered couple $e=(v_0,v_1)$ of two vertices, respectively the origin and the terminus of the edge $e$. In an analogous fashion, we define a directed $n$-path $d$, as a sequence of $n+1$ vertices $d=(v_0,\ldots,v_{n})$ such that for each $i=0,\ldots n$, $(v_i,v_{i+1})\in \mathcal{E}(\mathcal{T})$; we denote the origin and terminus of $d$ by, respectively $o(d)=v_0$ and $t(d)=v_n$. If moreover a directed $n$-path $d$ satisfies the condition $v_i\neq v_{i+2}$, we say that $d$ is a \emph{geodesic of length $n$} on $\mathcal{T}$. We denote the set of all geodesics of length $n$ by $Geod_{n}(\mathcal{T})$. Obviously, $\mathcal{E}(\mathcal{T})=Geod_1(\mathcal{T})$ and $\mathcal{V}(\mathcal{T})=Geod_0(\mathcal{T})$. Since $\mathcal{T}$ is a tree, a geodesic is a path without backtracking and it is uniquely determined by its origin and terminus, that is, by an ordered couple of vertices. We have the action by left multiplication of $PGL_2(\qp)$ on the vertices of $\mathcal{T}$, which extends to an action on the set of geodesics $Geod_n(\mathcal{T})$. It is well known that the action is transitive on $\mathcal{V}(\mathcal{T})$ and $\mathcal{E}(\mathcal{T})$, but the same holds true for $Geod_n(\mathcal{T})$. Describing each geodesic as the couple consisting of origin and terminus point, it is clear that the stabilizer of a geodesic $g$ is
	\begin{equation*}
		Stab_{PGL_2(\qp)}(g)=Stab_{PGL_2(\qp)}(o(g))\cap Stab_{PGL_2(\qp)}(t(g))
	\end{equation*}
	\emph{i.e.} a matrix stabilizes $g$ if and only if it stabilizes both its origin and terminus. We fix, for each $n\geq 0$, a privileged geodesic of length $n$ (we call it privileged vertex and edge, respectively in the case $n=0$ and $n=1$) namely
	\begin{equation*}
		g_n=\left(\mat{1}{0}{0}{1}\cdot PGL_2(\zp), \mat{1}{0}{0}{p}\cdot PGL_2(\zp),\ldots, \mat{1}{0}{0}{p^n}\cdot PGL_2(\zp)\right).
	\end{equation*}
	For ease of notation we often write $g_n=\left(\mat{1}{0}{0}{1}, \mat{1}{0}{0}{p},\ldots, \mat{1}{0}{0}{p^n}\right)$, as well for generic geodesics, where we write a choice of representatives instead of the classes. It is readily computed that
	\begin{equation*}
		Stab_{PGL_2(\qp)}\left(\mat{1}{0}{0}{1}\cdot PGL_2(\zp)\right)=PGL_2(\zp)
	\end{equation*}
	as well as that
	\begin{equation*}
		Stab_{PGL_2(\qp)}\left(\mat{1}{0}{0}{p^n}\cdot PGL_2(\zp)\right)=\mat{1}{0}{0}{p^n}PGL_2(\qp)\mat{1}{0}{0}{p^n}^{-1},
	\end{equation*}
	thus the characterization of the stabilizers of a geodesics implies that the stabilizer of the privileged geodesic $g_n$ is
	\begin{equation*}
		Stab_{PGL_2(\qp)}(g_n)=PGL_2(\qp)\cap\mat{1}{0}{0}{p^n}PGL_2(\qp)\mat{1}{0}{0}{p^n}^{-1}\\
		=\left\{\mat{a}{b}{c}{d} \in PGL_2(\zp)\mid c\in p^n\zp\right\}.
	\end{equation*}
	We denote the above stabilizer by $\overline{\Gamma}_0(p^n\zp)$ as it is the image in $PGL_2(\qp)$ of
	\begin{equation*}
		\Gamma_0(p^n\zp)=\left\{\mat{a}{b}{c}{d} \in GL_2(\zp)\mid c\in p^n\zp\right\}.
	\end{equation*}
	Therefore, the quotient map induces the identification
	\begin{equation*}
		PGL_2(\qp)/\overline{\Gamma}_0(p^n\zp) = 	GL_2(\qp)/\qp^\times \Gamma_0(p^r\zp),
	\end{equation*}
	and since the $PGL_2(\qp)$-action is transitive, we obtain the isomorphism
	\begin{align}\label{geodesics iso}
		GL_2(\qp)/\qp^\times \Gamma_0(p^r\zp)\cong 	Geod_{r}(\mathcal{T})
	\end{align}
	given by
	\begin{equation*}
		\gamma \longmapsto \gamma\cdot g_n= \left(\gamma\cdot\mat{1}{0}{0}{1}, \gamma\cdot\mat{1}{0}{0}{p},\ldots, \gamma\cdot\mat{1}{0}{0}{p^n}\right).
	\end{equation*}
	The inclusion of $\Gamma_0(p^n\zp)$ into $\Gamma_0(p^{m}\zp)$ for each $m\leq n$ provides the natural ($GL_2(\qp)$-equivariant) quotient map
	\begin{equation*}
		\rho_n: GL_2(\qp)/\qp^\times \Gamma_0(p^n\zp)\longrightarrow GL_2(\qp)/\qp^\times 	\Gamma_0(p^{n-1}\zp).
	\end{equation*}
	We can consider the surjective map
	\begin{equation*}
		\alpha_n: 	Geod_{n}(\mathcal{T})\longrightarrow Geod_{n-1}(\mathcal{T})
	\end{equation*}
	defined by
	\begin{equation*}
		(v_0,\ldots,v_{n-1},v_{n})\longmapsto (v_0,\ldots,v_{n-1}),
	\end{equation*}
	\emph{i.e.} the map extracting the initial geodesic of length $n-1$. We note that this map is $GL_2(\qp)$-equivariant. 
	\begin{lemma}\label{proj maps}
		The diagrams
		%	\begin{center}
			\begin{equation*}
				\begin{tikzcd}
					\vdots \arrow[d,two heads, dashed ,"\rho_{n+1}"] && \vdots \arrow[d,two heads, dashed 	,"\alpha_{n+1}"]\\
					GL_2(\qp)/\qp^\times \Gamma_0(p^n\zp)\arrow[rr, leftrightarrow,"\cong"]\arrow[d,two 	heads,"\rho_n"] && Geod_{n}(\mathcal{T})\arrow[d,two heads,"\alpha_n"]\\
					GL_2(\qp)/\qp^\times \Gamma_0(p^{n-1}\zp)\arrow[rr, leftrightarrow,"\cong"] \arrow[d,two heads, dashed ,"\rho_{n-1}"] && Geod_{n-1}(\mathcal{T}) \arrow[d,two heads, dashed 	,"\alpha_{n-1}"]\\
					\vdots \arrow[d,two heads, dashed 	,"\rho_{2}"] && \vdots \arrow[d,two heads, dashed ,"\alpha_{2}"]\\
					GL_2(\qp)/\qp^\times 	\Gamma_0(p\zp)\arrow[rr, leftrightarrow,"\cong"]\arrow[d,two heads,"\rho_1"] && \mathcal{E}(\mathcal{T})\arrow[d,two heads,"\alpha_1"]\\
					GL_2(\qp)/\qp^\times GL_2(\zp) 	\arrow[rr, leftrightarrow,"\cong"] && \mathcal{V}(\mathcal{T})
				\end{tikzcd}
			\end{equation*}
			%	\end{center}
		are commutative.
	\end{lemma}
	\begin{proof}
		Let $n>1$ and take $\bar{A} \in GL_2(\qp)/\qp^\times \Gamma_0(p^n\zp)$ and $\bar{B}\in GL_2(\qp)/\qp^\times \Gamma_0(p^{n-1}\zp)$ such that $\rho_n(\bar{A})=\bar{B}$. This means that there exists a matrix $C\in \overline{\Gamma}_0(p^{n-1}\zp)=Stab_{PGL_2(\qp)}(g_{n-1})$ such that $A=BC$ as matrices in $PGL_2(\qp)$. Thus,
		%	\begin{center}
			\begin{equation*}
				\begin{tikzcd}
					\bar{A}=\bar{B}\bar{C} \arrow[rr, leftrightarrow]\arrow[d,mapsto,"\rho_n"] && 	\left(A\cdot\mat{1}{0}{0}{1},A\cdot\mat{1}{0}{0}{p},\ldots,A\cdot\mat{1}{0}{0}{p^{n-1}}, A\cdot\mat{1}{0}{0}{p^n}\right) \arrow[d,mapsto,"\alpha_n"]\\
					\bar{B}\arrow[rr, leftrightarrow] && \left(A\cdot\mat{1}{0}{0}{1}, 	A\cdot\mat{1}{0}{0}{p},\ldots,A\cdot\mat{1}{0}{0}{p^{n-1}}\right)
				\end{tikzcd}
			\end{equation*}
			%	\end{center}
		because
		\begin{equation*}
			\left(B\cdot\mat{1}{0}{0}{1}, B\cdot\mat{1}{0}{0}{p},\ldots,B\cdot\mat{1}{0}{0}{p^{n-1}}\right)=	\left(A\cdot\mat{1}{0}{0}{1}, A\cdot\mat{1}{0}{0}{p},\ldots,A\cdot\mat{1}{0}{0}{p^{n-1}}\right)
		\end{equation*}
		as $\alpha_n( A\cdot g_n)=A\cdot\alpha_n( g_n)=A\cdot g_{n-1}= B\cdot (C\cdot g_{n-1}) = B\cdot g_{n-1}$. The proof for the case $n=1$ is the same, taking into account that $PGL_2(\zp)$ stabilizes $g_0=\left(\mat{1}{0}{0}{1}\right)$. 
	\end{proof}
	\begin{remark}
		Since the Bruhat--Tits tree is $p+1$-regular, we have exactly $p+1$ preimages via $\alpha_1$ and $p$ preimages via $\alpha_n$, for each $n\geq 2$. Hence, we readily note that, for each $n\geq 2$,
		\begin{equation*}
			[\Gamma_0(p^{n-1}\zp):\Gamma_0(p^n\zp)]=[\overline{\Gamma}_0(p^{n-1}\zp): \overline{\Gamma}_0(p^n\zp)]=p
		\end{equation*}
		and
		\begin{equation*}
			[GL_2(\zp):\Gamma_0(p\zp)]=[PGL_2(\zp): \overline{\Gamma}_0(p\zp)]=p+1.
		\end{equation*}
		In particular, for each fixed vertex $v$ in $\mathcal{T}$, the set $Geod_n(v)$ of \emph{geodesic of length $n$ and origin $v$} is the preimage of $v$ via the composition $\alpha_n\circ\cdots\circ\alpha_1$ and thus it has cardinality $\# 	Geod_n(v)=[GL_2(\zp):\Gamma_0(p^n\zp)]=(p+1)p^{n-1}.$			
    \end{remark}
	
	\subsection{Recalls on quaternionic modular forms and Brandt matrices with characters}
	
	In order to explicitly compute the values of quaternionic modular forms with character we need the notion of Brandt matrices. In this section we recall the needed notions about quaternionic orders and Brandt matrices.
	
	We begin by fixing, once and for all, a choice of field embeddings $\overline{\q}\hookrightarrow\overline{\qp}\hookrightarrow \cc$. Moreover, for any prime $q$ we denote the $q$-adic valuation by $v_q$. In this section, let $B$ be the unique (up to isomorphism) quaternion algebra over $\q$ with discriminant $D$. For each prime $q$ and any positive integer $N$ prime to $D$, we set
	\begin{align*}
		\Gamma_0(N\z_q)=\left\{\gamma\in M_2(\z_q)\mid \gamma\equiv \mat{*}{*}{0}{*} \pmod{N\z_q}\right\}=\Gamma_0(q^{v_q(N)}\z_q).
	\end{align*}
	Let $R$ be an Eichler order of level $N$ and we fix isomorphisms $\iota_q:B_q=B\otimes \q_q\cong M_2(\q_q)$ for any $q\nmid D$ such that $\iota_q(R\otimes \z_q)=\Gamma_0(N\z_q)$. 
	\begin{notation}
		As above, we set $B_q=B\otimes \q_q$, $R_q=R\otimes \z_q$ for any order $R$, $\widehat{B}=B\otimes_\q \A_{\q,f}$ and $\widehat{R}=R\otimes_\z \widehat{\z}$ where $\A_{\q,f}=\q \widehat{\z}$ are the finite ad\`{e}les of $\q$. Similarly, we set $B(\A_\q)^\times=(B\otimes_\q\A_\q)^\times$ and $ R(\A_\q)^\times=\{r\in B(\A_\q)^\times\mid (r_q)_{q<\infty}\in \widehat{R}^\times\}$.
	\end{notation}
	
	\subsubsection{Characters}\label{lift character}
	
	Let $R$ be as above and consider $\chi$ to be a Dirichlet character of conductor $C$, where $C$ is prime to the discriminant $D$ and such that $C\mid N$. We want to extend $\chi$ to a character $\widetilde{\chi}$ of $\widehat{R}$ as in Section 7.2 of \cite{HPS1989orders}, which we refer to for further details about the general case. First of all, we decompose $\chi=\prod_{q\mid C}\chi_q$ with the Chinese Reminder Theorem and we define each character of $R_q$, $\widetilde{\chi_q}$, as $\widetilde{\chi_q}(\alpha)=\chi_q(d)$ for $\alpha=\left(\begin{smallmatrix}
		a & b\\
		c & d
	\end{smallmatrix}\right)\in R_q=\Gamma_0(N\z_q)$. Patching together the local lifts, we define
	\begin{equation*}
		\widetilde{\chi}(b)=\prod_{q\mid C}\widetilde{\chi_q}(b_q)
	\end{equation*}
	for $b\in B(\A_\q)^\times$ such that $b_q\in R_q^\times\subset B_q^\times$. In particular, if $I$ is a lattice in $B$ such that $I_q=I\otimes\z_q=R_q$ for each $q|C$, and $b\in I$, we have
	\begin{equation*}
		\widetilde{\chi}(b)=\prod_{q\mid C}\widetilde{\chi_q}(b).
	\end{equation*}
	
	\subsubsection{Quaternionic modular forms}\label{Quaternionic modular forms of weight $2$}
	As we are interested only in considering weight-2 modular forms, we are recalling their definition and the few properties we need.
	\begin{definition}\label{quaternionic mf def}
		One defines the space of weight-$2$ quaternionic modular forms with level structure $R(\A_\q)^\times$, character $\chi$ and $\cc$-coefficients, as the $\cc$-vector space $S_2(R, \widetilde{\chi})$ of all continuous functions
		\begin{equation*}
			\varphi:B(\A_\q)^\times\longrightarrow\cc
		\end{equation*}
		satisfying
		\begin{equation*}
			\varphi(b \tilde{b} r)=\widetilde{\chi}^{-1}(r)\varphi(\tilde{b})
		\end{equation*}
		for all $b\in B^\times$, $\tilde{b}\in B(\A_\q)^\times$ and $r\in R(\A_\q)^\times$.
	\end{definition}
	\noindent As in Chapter 5 of \cite{HPS1989basis}, we can decompose $B(\A_\q)^\times$ as a finite union of distinct double-cosets
	\begin{equation*}
		B(\A_\q)^\times=\coprod_{i=1}^{h} B^\times x_i R(\A_\q)^\times
	\end{equation*}
	where $h=h(R)$ is the class number of $R$. Since $B$ is definite, the analogous decomposition holds for $\widehat{B}^\times$, namely $\widehat{B}^\times=\coprod_{i=1}^{h}B^\times \widehat{x_i} \widehat{R}^\times$, with $\widehat{x_i}=(x_{i,q})_{q<\infty}$.
	We recall that the above double quotient can be identified with the class group of $R$, that is classes of left $R$-ideals.
	\begin{lemma}[\cite{HPS1989orders}, Lemma 7.4]\label{HPS1989orders Lemma 7.4}
		Let $R$ be an Eichler order of level $N$. Then there exists a set of ideal class representatives $\{I_1,\ldots,I_h\}$ for the left $R$-ideal classes, such that $I_i\otimes\z_q=R_q$ for all $q$ dividing the level.
	\end{lemma}
	\noindent By the above Lemma \ref{HPS1989orders Lemma 7.4}, the representatives $x_i=(x_{i,q})_{q}\in B(\A_\q)^\times$ can be taken to lie in $R(\A_\q)^\times$, in particular $x_{i,q}\in R_q^\times$ for each prime $q|N$. Fixing the representatives in this fashion, we readily notice that the lift of the characters in Section \ref{lift character} makes sense. Furthermore, the double-coset decomposition also allows for a more explicit expression of quaternionic modular forms. By the definition of a quaternionic modular forms, a quaternionic modular form $\varphi$ is uniquely determined by its values on the representatives $x_i$. More precisely, for $i=1,\ldots, h$, let $\widetilde{\Gamma}_{x_i}= B^\times \cap x_i^{-1}R^\times x_i$ and define
	\begin{equation*}
		\cc_{\widetilde{\chi},i}=\left\{c\in\cc\mid \widetilde{\chi}(\gamma)\cdot c=c,\textrm{ for each }\gamma\in \widetilde{\Gamma}_{x_i}\right\}.
	\end{equation*}
	As thoroughly explained in Chapter 5 of \cite{HPS1989basis}, the above observations yield the identification
	\begin{equation}\label{quat mod forms as values}
		S_2(R, \widetilde{\chi})\cong \bigoplus_{i=1}^{h}\cc_{\widetilde{\chi},i}
	\end{equation}
	given by $\varphi\longmapsto (\varphi(x_1),\ldots,\varphi(x_h))$.
	\begin{remark}\label{inclusion and stabilizers}
		\begin{enumerate}[label=(\alph*)]
			\item It is possible to fix compatible choices of the lifting characters $\widetilde{\chi}$ such that the liftings are compatible under the inclusion of orders.
			\item For level $N$ big enough, the groups $\widetilde{\Gamma}_{x_i}$ have cardinality $2$, \emph{i.e.} $\widetilde{\Gamma}_{x_i}=\{\pm 1 \}$.
		\end{enumerate}
	\end{remark}
	
	\subsubsection{Brandt Matrices}\label{Brandt}
	
	We begin this section by denoting the reduced norm of $b\in B$ by $n(b)\in\q$ and, by abuse of notation, we will keep the notation for the reduced norm of $\widehat{B}$. Let $\{I_1=R,\ldots,I_h\}$ be a set of representatives for the classes of left $R$-ideals and $R_i$, for $i=1,\ldots,h$, the right order corresponding to $I_i$. For $1\leq i,j \leq h$, we set $M_{i,j}=I_j^{-1}I_i$ which is a $R_j$-left ideal with right order $R_i$. We set $n(M_{i,j})$ to be the unique positive rational number such that $n(b)/n(M_{i,j})$ are all integers with no common factor, for $b$ varying in $M_{i,j}$. Take $\chi$ and $\widetilde{\chi}$ as in the above sections.
	\begin{definition}
		One defines the \emph{$m$-th Brandt matrix associated with the order $R$ and with character $\chi$}, as the matrix
		\begin{equation*}
			B(m,R,\chi)=\left(b_{i,j}(m,R,\chi)\right)_{1\leq i,j\leq h}, 
		\end{equation*}
		where
		\begin{equation*}
			b_{i,j}(m,R,\chi)=\frac{1}{2}\sum_{\substack{b\in M_{i,j}:\\ n(b)=m\cdot 	n(M_{i,j})}}\widetilde{\chi}(b)w_j^{-1}\ph{000}\text{and}\ph{000}b_{i,j}(0,R,\chi)=\begin{cases}
				\frac{1}{2w_j} & \text{if $\chi$ is trivial}\\
				0	& \text{otherwise}
			\end{cases}
		\end{equation*}
		and $w_i=\# (R_i^\times/\{\pm1\})$.
	\end{definition}
	\noindent By Lemma \ref{HPS1989orders Lemma 7.4}, the evaluation of a Dirichlet character at $b\in M_{i,j}$ does not present any ambiguity; the lemma guarantees that, up to a suitable choice of the representatives $I_i$, we can define the Brandt matrices as above.
	\begin{proposition}[\cite{HPS1989basis}, Theorem 4.8]
		With the above notation and for any couple of fixed indexes $(i,j)$, the expansion
		\begin{equation*}
			\theta_{i,j}(z)=\sum_{m=0}^{\infty}b_{i,j}(m,R,\chi)e^{2\pi i mz}
		\end{equation*}
		defines a modular forms in $M_2(\Gamma_0(ND),\chi,\cc)$. Moreover, the form is cuspidal if $\chi_q$ is non trivial for some $q|N$.
	\end{proposition}
	
	\subsubsection{Explicit computation of quaternionic modular forms}\label{brandt matrices and explicit quaternionic modular forms}
	
	We end this section by recalling how one can explicitly compute quaternionic modular forms (of weight 2) on definite quaternion algebras. The procedure and algorithms can be traced back to \cite{Eichler73} and \cite{Pizer1980}. To explain them, we need to recall two results.
	\begin{proposition}[Proposition 2.22, \cite{Pizer1980}, Theorems 2 \& 4, \cite{Eichler73}]
	    As above, let $R$ be an Eichler order of level $N$ and $\chi$ a Dirichlet character. The Brandt matrices $B(m,R,\chi)$, for $(m,N)=1$, form a commutative semisimple ring. Therefore, they can be simultaneously diagonalized by a matrix $C(R,\chi)$.
	\end{proposition}
	\begin{remark}[Remark 2.27, \cite{Pizer1980}]
	    The $h$ entries of the diagonal matrix
	    \begin{equation*}
	        D(R,\chi,z)=\sum_{m=1}^{+\infty}C(R,\chi) B(m,R,\chi) C(R,\chi)^{-1} e^{2\pi i m z},
	    \end{equation*}
	    comprehend of $h-1$ eigenforms in $S_2(\Gamma_1(ND))^{D-new}$. The quaternionic modular newforms associated via Jacquet--Langlands with these classical forms appear as columns of the $h\times h$ matrix $C(R,\chi)^{-1}$. Each entry of a chosen column represents the value of the form on a representative as in Equation \ref{quat mod forms as values}. The remaining column corresponds to the Eisenstein form in the remaining entry on the diagonal of $D(R,\chi,z)$. Further details can be found in Section 5 of \cite{HPS1989basis}.
	\end{remark}
	\begin{lemma}[Sturm bound, Corollary 9.20, \cite{stein2007}]
	    With the above notation, let $f\in S_k(\Gamma_1(ND),\chi)^{new}$ be a normalized eigenform. Then, $f$ is uniquely determined by its first $r$ Fourier coefficients, for
	    \begin{equation*}
	        r= \left\lfloor \frac{k [SL_2(\z):\Gamma_0(ND)]}{12}\right\rfloor=\left\lfloor \frac{kND}{12}\prod_{q|ND}\left(1+\frac{1}{q}\right)\right\rfloor.
	    \end{equation*}
	    The number $r$ is called the \emph{Sturm bound} of $S_k(\Gamma_1(ND),\chi)$.
	\end{lemma}
	\begin{example}
	    The Sturm bound of $S_2(\Gamma_0(15))$ is $\left\lfloor \frac{30}{12} \left(1+\frac{1}{3}\right)\left(1+\frac{1}{5}\right)\right\rfloor=4$.
	\end{example}
	
	Combining the above observations we deduce that a finite number of Brandt matrices are sufficient to determine the space of cuspidal quaternionic modular forms. We report here (see \cite{Pizer1980}) the routine for computing a quaternionic modular form associated with a normalized eigenform $f\in S_2(\Gamma_1(ND),\chi)^{new}$.
	\begin{routine}\label{alg routine quaternionic mod forms}\phantom{0}
    	\begin{enumerate}[label ={\emph{\textbf{Step \arabic*:}}},align=left]
    	    \item Compute the Sturm bound $r$ of $S_2(\Gamma_1(ND)$ and the eigenvalues with index lesser equal than this number.
    	    \item Construct an Eichler order $R$ of level $N$ (see the following Section \ref{basis for orders}).
    	    \item Compute Brandt matrices $B(m,R,\chi)$ for $m\leq r$ and find a matrix $C(R\chi)$ which simultaneously diagonalize them (using Pizer's algorithms in \cite{Pizer1980}).
    	    \item Express $f$ as a linear combination of the basis of eigenforms appearing on the diagonal of $D(R,\chi,z)$ (again, it is enough to deal with the first $r$ coefficients).
    	    \item The corresponding combination of the columns of $C(R,\chi)^{-1}$ represents the values of the sought-for quaternionic modular form.
    	\end{enumerate}
	\end{routine}

\section{Quaternionic double quotients and geodesics on the Bruhat--Tits tree}\label{Quaternionic double quotients and geodesics on the Bruhat--Tits tree}
	
	In this section we describe how to explicitly compute a set of representatives for the double quotient of the quaternion algebra $B$. We show how the procedure presented in \cite{FrancMasdeu2014} can be extended and applied to our case of interest.
	
	\subsection{Quaternionic $p$-adic double-coset spaces}\label{p-adic double-coset spaces}
	Let $p$ and $\ell$ be two distinct odd primes and take $B$ to be the (unique up to isomorphism) quaternion algebra over $\q$ ramified exactly at $\ell$ and $\infty$. For $n\geq 0$, consider $R^n$ to be an Eichler order of level $Np^n$, with $N$ prime to both $p$ and $\ell$. We fix, as above,  isomorphisms $\iota_q:B_q=B\otimes \q_q\cong M_2(\q_q)$ for each prime $q\neq \ell$, such that $\iota_q:R_q =R\otimes \z_q\cong \Gamma_0(N p^n\z_q)$ and hence $\iota_p(R)=\Gamma_0(p^n\zp)$.
	
	We begin by recalling the classical application of the \emph{Strong Approximation Theorem} (see \cite{Vigneras1980}, Theoreme fondamental 1.4, b)).
	\begin{lemma}[$p$-adic double quotient]\label{p-adic model}
		Let $\Sigma=\prod'\Sigma_q$ be $\widehat{R^n}^\times$ and let $R$ be a maximal order containing $R^n$. The embedding of $GL_2(\qp)\overset{\iota_p}{\cong} B_p^\times\hookrightarrow \widehat{B}^\times$ as $b_p\longmapsto (1,\ldots,1,b_p,1,1\ldots)$ induces the bijection
		\begin{equation*}
			\left({R}\left[1/p\right]^\times\cap \prod^{}_{q\neq p}\Sigma_q\right)\backslash B_p^\times / \Sigma_p \cong B^\times\backslash \widehat{B}^\times / \Sigma.
		\end{equation*}
		Furthermore, we have
		\begin{equation*}
			B^\times\backslash \widehat{B}^\times / \Sigma \cong  \iota_p\left(R^n\left[1/p\right]^\times\right)\backslash GL_2(\qp) / \Gamma_0(p^n\zp)\cong  \Gamma_n\backslash GL_2(\qp) / \qp^\times \Gamma_0(p^n\zp),
		\end{equation*}
		where $\Gamma_n$ is the image of $\iota_p\left(R^n\left[1/p\right]^\times\right)$ in $PGL_2(\qp)$.
	\end{lemma}
	\begin{proof}
		Since $p$ is a split place in $B$, the \emph{Strong Approximation Theorem}\commenttt{tells us that $B^1 B_p^1$ is dense in $\widehat{B}^1$, for $(-)^1$ the subset of norm one elements. Since the reduced norm map is surjective away from the archimedean place, we have that $\widehat{B}^\times = \overline{(B^1 B_p^1 \A_{\q,f}^\times )}$.} implies that for any $\Sigma$ compact open in $\widehat{B}^\times$, $\widehat{B}^\times = B^1 B_p^1 \A_{\q,f}^\times \Sigma = B^\times B_p^\times \Sigma$.\commenttt{ guarantees the decomposition $\widehat{B}^\times = B^1 B_p^1 \A_{\q,f}^\times \Sigma = B^\times B_p^\times \Sigma$, as $\A_{\q,f}^\times\subset B^\times B_p^\times \Sigma$. It remains to prove that $B^\times\cap B_p^\times \Sigma = {R}\left[1/p\right]\cap \prod^{}_{q\neq p}\Sigma_q$ and that $B^\times B_p^\times\cap \Sigma=\Sigma_p$. Taking in account that $B^\times$ embeds diagonally in $\widehat{B}^\times$ and that $B_p^\times\hookrightarrow \widehat{B}^\times$ as $b_p\longmapsto (1,\ldots,1,b_p,1,1\ldots)$, the two intersection are easily proved. The isomorphism theorem proves the first statement.\\}
		By proving the double inclusions we have ${R} \left[1/p\right]^{\times} \cap \prod_{q\neq p} {(R^n_q)}^\times = R^n\left[1/p\right]^{\times}$. In the end, $1/p$ belongs to ${R^n[1/p]}^\times$ and $\zp^\times\subset {R^n_p}^\times$.
		\commenttt{The element $\frac{1}{p}$ belongs to ${R}[1/p]$ and to each $R^n_q$ for $q\neq p$. By definition $\zp\subset R^n_p$, hence $\qp^\times$ is contained in $\widehat{R^n}\cap \prod^{}_{q\neq p}{R}_q$. The center of $B_p^\times$ is exactly $\qp^\times$ and the reduced norm is surjective (with $\mat{\zp^\times}{0}{0}{\zp^\times}\subset (R^n_p)^\times$, by definition of $(R^n_p)^\times$), thus
			\begin{equation*}
				\left({R}\left[1/p\right]^\times\cap \prod_{q\neq p} (R^n_q)^\times\right) \backslash B_p^\times / (R^n_p)^\times\cong \left({R}\left[1/p\right]^\times\cap \prod_{q\neq p} (R^n_q)^\times\right)^{1} \backslash B_p^\times / \qp^\times (R^n_p)^\times.
			\end{equation*}
			The last step is noticing that ${R} \left[1/p\right]^{\times} \cap \prod_{q\neq p} (R^n_q)^\times = R^n\left[1/p\right]^{\times}$.}
	\end{proof}
	\noindent The above lemma, together with the following proposition, shows that we can focus our attention on the right-quotient
	\begin{equation}\label{iota p iso}
		B_p^\times / \qp^\times 	R^\times_{p}\overset{\iota_p}{\cong} GL_2(\qp)/\qp^\times \Gamma_0(p^r\zp).
	\end{equation}
	\begin{proposition}\label{left inv order}
		For any $n\geq 1$, let $R^n$ be as above. Up to conjugation, we can assume that the orders $R_n$ are encapsulated. Then $R^n\left[1/p\right]=R^{n+1}\left[1/p\right]$.
	\end{proposition}
	\begin{proof}
		Let $f_1,\ldots, f_4$ be a $\z$-basis for $R^n$. As thoroughly explained in \cite{Pizer1980}, Section 5, any sublattice (in particular suborders) of index $p$ in $R^n$, has a basis $(g_1,\ldots,g_4)$ such that $(g_1,\ldots,g_4)=(f_1,\ldots,f_4)\cdot A$, where $A$ varies among the matrices in Hermite normal form in $M_4(\z)$ and with determinant $p$. Such matrices are of the form
		\begin{align*}
			\left(\begin{matrix}
				p & 0 & 0 & 0\\
				0 & 1 & 0 & 0\\
				0 & 0 & 1 & 0\\
				0 & 0 & 0 & 1
			\end{matrix}\right),\ph{000}
			\left(\begin{matrix}
				1 & 0 & 0 & 0\\
				a & p & 0 & 0\\
				0 & 0 & 1 & 0\\
				0 & 0 & 0 & 1
			\end{matrix}\right),\ph{000}
			\left(\begin{matrix}
				1 & 0 & 0 & 0\\
				0 & 1 & 0 & 0\\
				a & b & p & 0\\
				0 & 0 & 0 & 1
			\end{matrix}\right),\ph{000}
			\left(\begin{matrix}
				1 & 0 & 0 & 0\\
				0 & 1 & 0 & 0\\
				0 & 0 & 1 & 0\\
				a & b & c & p
			\end{matrix}\right),
		\end{align*}
		for $0\leq a,b,c<p$. These matrices belong to $M_4(\z)\cap GL_4\left(\z\left[1/p\right]\right)$, thus the claim follows.
	\end{proof}
	\begin{notation}
		We deduce that $\Gamma_n=\Gamma_{n+1}$ for each $n\geq 1$, hence we set $\Gamma=\Gamma_1$.
	\end{notation}
	% 	\begin{remark}
		% 	    Even though we only deal with level structure associated with Eichler orders, we point out that the results in Section \ref{Quaternionic double quotients and geodesics on the Bruhat--Tits tree}, and especially those contained in Section \ref{p-adic double-coset spaces}, are still valid in the case of special orders defined by \cite{Pizer80p2} and \cite{HPS1989orders}.
		% 	\end{remark}
	
\subsection{Classes of representatives for the finite-length geodesics}\label{classes of representatives}
	We recall that, for any prime $q$, we denote the $q$-adic valuation by $v_q$.
	\begin{lemma}\label{coset matrices}
		For $r\geq 1$, the quotient $GL_2(\qp)/\qp^\times \Gamma_0(p^r\zp)$ admits a set of right-coset representatives, which we denote by $\left\{e_i\right\}$, consisting of matrices with coefficients in $\z$ if $r=1$, and coefficients in $\zp$ if $r\geq 2$. Moreover, there exists an effective algorithm that, given a matrix $g$ in $GL_2(\qp)$ returns a scalar $\lambda\in\qp^\times$ and a matrix $t$ in $\Gamma_0(p^r\zp)$, such that $\lambda\cdot g \cdot t= e_i$.
	\end{lemma}
	Before giving a proof of the lemma, which provides the actual algorithm for computing $\lambda$ and $t$, we point out that the case of $r=1$ has already been proved in Lemma 2.2 of \cite{FrancMasdeu2014}. Lemma \ref{coset matrices} indeed agrees with the result in \emph{loc.cit.} and the computations are essentially the same performed there. Moreover, the case of $GL_2(\zp)$ is a classical computation and we refer to Section 5.3 of \cite{Miyake2006}.
	\begin{proof}[Proof of Lemma \ref{coset matrices}]
		Let $M=\mat{a}{b}{c}{d}$ be a matrix in $GL_2(\qp)$. Up to multiply by an element in $\qp^\times$, we can suppose $M\in GL_2(\qp)\cap M_2(\zp)$. We set $\lambda$ to be the power of $p$ such that $\lambda M$ has coefficients in $\zp$ and at least one entry is in $\zp^\times$. Set $\lambda=p^{-\min\{v_p(a),v_p(b),v_p(c),v_p(d)\}}$. 
		Let $N=v_p(\det(\lambda M))$ and from now on we assume that $M$ is such that $\min\{v_p(a),v_p(b),v_p(c),v_p(d)\}=0$. We divide the proof into three main cases. The proof is a straightforward computation once we provide a suitable matrix $t$, which is a product of the three following (usual) transformations:
		\begin{align*}
			\mat{a}{b}{c}{d}\cdot \mat{1}{x}{0}{1}=\mat{a}{ax+b}{c}{cx+d},&&\mat{a}{b}{c}{d}\cdot \mat{u}{0}{0}{v}=\mat{au}{bv}{cu}{dv},&&		\mat{a}{b}{c}{d}\cdot \mat{1}{0}{yp^r}{1}=\mat{a+byp^r}{b}{c+dyp^r}{d}, 
		\end{align*}
		for $\mat{1}{x}{0}{1},\mat{u}{0}{0}{v}\in\Gamma_0(p^N\zp)$ for each $N\geq 0$ and $\mat{1}{0}{yp^r}{1}\in \Gamma_0(p^r\zp)$. We remark that the $p$-adic valuation of the determinant of these transformations is zero.
		\begin{enumerate}[topsep=\baselineskip]
			\item $v_p(a)\leq v_p(b)$: This first case is proved in the same fashion as in \cite{FrancMasdeu2014}.  We have 
			\begin{equation*}
				t=\Mat{1}{-\frac{b}{a}}{0}{1}\cdot \Mat{\frac{p^{v_p(a)}}{a}}{0}{0}{1}\cdot \Mat{1}{0}{0}{\frac{p^{N-\alpha}}{d-c\frac{a}{b}}}\cdot 	\Mat{1}{0}{yp^r}{1}
			\end{equation*}
			for $y\in\zp$ such that $\frac{p^{v_p(a)}}{a}c+yp^r$ belongs to $\z\cap (0,p^{r+N-v_p(a)}]$. 
			Take $s\in\z$ such that $s\cdot \frac{a}{p^{v_p(a)}}\equiv 1 \pmod{p^{r+N-v_p(a)}}$ and $c'\in\z\cap (0,p^{r+N-v_p(a)}]$ such that $c'\equiv sc\pmod{p^{r+N-v_p(a)}}$. Thus
			\begin{equation*}
				\Mat{a}{b}{c}{d}\equiv \Mat{p^{v_p(a)}}{0}{c'}{p^{N-v_p(a)}}\pmod{\Gamma_0(p^r\zp)}. 
			\end{equation*}
			\item $v_p(a)\geq v_p(b)+r$: The matrix $t$ is of the following form
			\begin{equation*}
				t=\Mat{1}{0}{-\frac{a}{bp^r}p^r}{1}\cdot \Mat{\frac{p^{N-v_p(b)}}{c-\frac{a}{b}d}}{0}{0}{\frac{p^{v_p(b)}}{b}}\cdot\Mat{1}{x}{0}{1}
			\end{equation*}
			for $x\in\zp$ such that $\frac{p^{v_p(b)}}{b}d+xp^{N-v_p(b)}$ belongs to $\z\cap (0,p^{N-v_p(b)}]$.
			Take $s\in\z$ such that $s\cdot \frac{p^{v_p(b)}}{b}\equiv 1 \pmod{p^{N-v_p(b)}}$ and $d'\in\z\cap (0,p^{N-v_p(b)}]$ such that $d'\equiv sd\pmod{p^{N-v_p(b)}}$. Thus
			\begin{equation*}
				\Mat{a}{b}{c}{d}\equiv \Mat{0}{p^{v_p(b)}}{p^{N-v_p(b)}}{d'}\pmod{\Gamma_0(p^r\zp)}. 
			\end{equation*}
			\item $v_p(b)<v_p(a)< v_p(b)+r$: This case appears only when $r>1$. As above, we can distinguish three further cases, depending on the valuations $v_p(c)$ and $v_p(d)$.
			\begin{enumerate}[topsep=\baselineskip]
				\item $v_p(c)\leq v_p(d)$: We proceed as in the case $v_p(a)\leq v_p(b)$ and define a unique matrix $t\in\Gamma_0(p^r\zp)$, analogous to the one above, of the form
				\begin{equation*}
					t=\Mat{1}{-\frac{d}{c}}{0}{1}\cdot \Mat{*}{0}{0}{1}\cdot \Mat{1}{0}{0}{*}\cdot \Mat{1}{0}{yp^r}{1}.
				\end{equation*}
				Take $s\in\z$ such that $s\cdot \frac{p^{v_p(c)}}{c}\equiv 1 \pmod{p^{v_p(b)+r}}$ and $a'\in\z\cap (0,p^{r}]$ such that $a'\equiv sa/p^{v_p(b)}\pmod{p^{r}}$. Since $v_p(a)$ is bigger than $v_p(b)$, we deduce that
				\begin{equation*}
					\Mat{a}{b}{c}{d}\equiv \Mat{a'p^{v_p(b)}}{p^{v_p(b)}}{p^{N-v_p(b)}}{0}\pmod{\Gamma_0(p^r\zp)},
				\end{equation*}
				and, since we are supposing that one of the entries is in $\zp^\times$, we are in either one of the following two situations:
				\begin{equation*}
					\Mat{a'p^{v_p(b)}}{p^{v_p(b)}}{p^{N-v_p(b)}}{0}=\begin{cases}
						\Mat{a'}{1}{p^{N}}{0} & \text{ if }v_p(b)=0,\\
						\\			\Mat{a'p^{v_p(b)}}{p^{v_p(b)}}{1}{0} & \text{ if }N=v_p(b)\text{ (\emph{i.e.} if }v_p(b)\geq 1).
					\end{cases}
				\end{equation*}
				\item $v_p(c)\geq v_p(d)+r$: As in the case of $v_p(a)\geq v_p(b)+r$, we have
				\begin{equation*}
					t=\Mat{1}{0}{-\frac{c}{dp^r}p^r}{1}\cdot \Mat{*}{0}{0}{*}\cdot\Mat{1}{x}{0}{1}
				\end{equation*}
				for $x\in\zp$ such that $\frac{p^v_p(d)}{d}b+xp^{N-v_p(d)}$ belongs to $\z\cap (0,p^{N-v_p(d)}]$.
				Take $s\in\z$ such that $s\cdot \frac{p^v_p(d)}{d}\equiv 1 \pmod{p^{N-v_p(d)}}$ and $b'\in\z\cap (0,p^{N-v_p(d)}]$ such that $b'\equiv sb\pmod{p^{N-v_p(d)}}$. Thus
				\begin{equation*}
					\Mat{a}{b}{c}{d}\equiv \Mat{p^{N-v_p(d)}}{b'}{0}{p^{v_p(d)}}\pmod{\Gamma_0(p^r\zp)}. 
				\end{equation*}
				\item $v_p(d)<v_p(c)<v_p(d)+r$: This is the last case, the one in which we cannot say much about the reduction modulo $\Gamma_0(p^r\zp)$. We have
				\begin{equation*}
					0<v_p(a)-v_p(b),v_p(c)-v_p(d)<r
				\end{equation*}
				and $t$ is the diagonal matrix with $\zp^\times$ coefficients, $t=\mat{*}{0}{0}{*}$, hence
				\begin{equation*}
					\Mat{a}{b}{c}{d}\equiv \begin{cases}
						\Mat{p^{v_p(a)}}{1}{\frac{cp^{v_p(a)}}{ap^{v_p(c)}}\cdot p^{v_p(c)}}{\frac{d}{bp^{v_p(d)}}\cdot p^{v_p(d)}} 	\pmod{\Gamma_0(p^r\zp)} & \text{ if }v_p(d)\geq v_p(b),\\%\text{ iff }v_p(b)=0,\\
						\\	\Mat{\frac{ap^{v_p(c)}}{cp^{v_p(a)}}\cdot p^{v_p(a)}}{\frac{b}{dp^{v_p(b)}}\cdot p^{v_p(b)}}{p^{v_p(c)}}{1} 	\pmod{\Gamma_0(p^r\zp)} & \text{ if }v_p(d)<v_p(b).%\text{ iff }v_p(d)=0.
					\end{cases}
				\end{equation*}
			\end{enumerate}
		\end{enumerate}
	\end{proof}
	
	\noindent We can explicitly write the representatives produced in the proof of Lemma \ref{coset matrices} and restate the lemma in the following form.
	
	\begin{lemma}\label{coset matrices2}
		It exists an effective algorithm that, given a matrix $g$ in $GL_2(\qp)$ returns a unique scalar $\lambda\in p^{\z}$ and a unique matrix $t$ in $\Gamma_0(p^r\zp)$, such that $\lambda\cdot g \cdot t$ is, for $m,n\geq 0$, a matrix in one of the following forms:
		\begin{align*}
			\Mat{p^m}{0}{c}{p^n} &\text{ for } 0<c\leq p^{n+r}, &&  \Mat{0}{p^m}{p^n}{d} &\text{ for } 0<d\leq p^{n},\hspace{1.9cm}\\ 
			\Mat{ap^m}{p^m}{1}{0} &\text{ for } 0<a\leq p^{r},&&
			\Mat{a}{1}{p^n}{0} &\text{ for } 0<a\leq p^{r},\hspace{1.9cm}\\
			\Mat{p^m}{b}{0}{p^n} &\text{ for } 0<b\leq p^{m}, && & \\ 	\Mat{p^l}{1}{c'p^kp^n}{d'p^n} &\text{ for } 0<k,l<r, \,c',d'\in\zp^\times, && \Mat{a'p^lp^m}{b'p^m}{p^k}{1} &\text{ for } 0<k,l<r, \,a',b'\in\zp^\times.
		\end{align*}
	\end{lemma}
	$\\$
	
	\subsection{The algorithm for $\Gamma^1$-classes}\label{the algorithm}
	
	Let $B$, $R$ and $R^n$ as in Section \ref{p-adic double-coset spaces}. We identify the $p$-component of $R^n_{p}$ with $\Gamma_0(p^r\zp)$ via the fixed isomorphism $\iota_p$. We omit the isomorphism $\iota_p$ if there is no possibility of confusion. In the rest of this section we provide an immediate extension of Algorithm 1 contained in \cite{FrancMasdeu2014}; it allows us to compute a fundamental domain in the Bruhat--Tits tree $\mathcal{T}$ for the group $\Gamma^1=\left(R^n\left[1/p\right]^\times\right)^{1}$ of elements of reduced norm 1. In addition, Algorithm \ref{algorithm} computes a full set of representatives for the double quotient $\Gamma^1\backslash GL_2(\qp)/\qp^\times \Gamma_0(p^r\zp)$. As in \cite{FrancMasdeu2014},  $\Gamma^1$ has an increasing filtration by finite sets (because the quaternion algebra is definite), for each $m\geq 0$, given by
	\begin{equation*}
		\Gamma^1(m)=\left\{\frac{x}{p^{m}}\mid x\in R^n \text{ and }n(x)=p^{2m}\right\},
	\end{equation*}
	where, $n(x)$ is the reduced norm map of $B$ (over $\q$) applied to $x$. Moreover, $\Gamma^1(m)\subset \Gamma^1(m+1)$ (as $x/p^m=px/p^{m+1}$) and $\Gamma^1=\cup_{m\geq 0}\Gamma^1(m)$.
	
	Let $u$ and $v$ be two matrices in $PGL_2(\qp)$ which, by abuse of notation, we identify with two vertices on $\mathcal{T}$. Following \cite{FrancMasdeu2014} and \cite{BoeckleButenuth2012}, we define $Hom_{\Gamma^1}(u,v)=\left\{\gamma\in \Gamma^1\ph{.}\Big|\ph{.}\gamma u = v\right\}$, that is the set of elements moving $u$ to $v$. We note that, for each vertex $u$, $Hom_{\Gamma^1}(u,u)=Stab_{\Gamma^1}(u)$ and $u$ is equivalent to $v$ modulo $\Gamma^1$ if and only if $Hom_{\Gamma^1}(u,v)\neq \emptyset$. From now on, we suppose that the matrices $u$ and $v$ are in one of the forms given in the above Lemma \ref{coset matrices2}. Let $m\in\frac{1}{2}\z_{\geq 0}$ be defined as  $2m=v_p\left( \det(u)\det(v)\right)$. We note that
	\begin{equation*}
		Hom_{\Gamma^1}(u,v)=Hom_{GL_2(\qp)}(u,v)\cap \Gamma^1 =\left(v\cdot Stab_{GL_2(\qp)}(g_0) \cdot u^{-1}\right)\cap \Gamma^1
		=\left(v\cdot \qp^\times GL_2(\zp) \cdot u^{-1}\right)\cap \Gamma^1.
	\end{equation*}
	All the lemmas and definitions needed for implementing Algorithm 1 in \cite{FrancMasdeu2014} are still valid in our setting without any change. Here we recall some of the results and refer the reader to \emph{loc.cit.} for a detailed proof of them. 
	\begin{lemma}[\cite{FrancMasdeu2014}, Lemma 3.1]
		If $m$ is not an integer, then $Hom_{\Gamma^1}(u,v)= \emptyset$. Otherwise
		\begin{equation*}
			Hom_{\Gamma^1}(u,v)=p^{-m}v M_2(\zp)u^*\cap\Gamma^1,
		\end{equation*}
		for $u^*=\frac{1}{\det u}u^{-1}$.
	\end{lemma}
	\begin{definition}[\cite{FrancMasdeu2014}, Definition 3.2]
		Let $m\geq 0$ be an integer, $V$ and $W$ two finite dimensional $\qp$-vector spaces, and let $\Lambda_V\subseteq V$ and $\Lambda_W\subseteq W$ be $\zp$-lattices. Let $F:V\longrightarrow W$ be a $\qp$-linear map satisfying $f(\Lambda_V)\subseteq \Lambda_W$. Then an \emph{approximation of $f$ to precision $m$} is a $\qp$-linear map $g:V\longrightarrow W$ such that $g\equiv f \pmod{p^m}$ when restricted to $\Lambda_V$.
	\end{definition}
	\begin{lemma}[\cite{FrancMasdeu2014}, Lemma 3.3]
		Let $u$ and $v$ be matrices in $GL_2(\qp)\cap M_2(\zp)$ corresponding to two vertices of $\mathcal{T}$. Let $f:M_2(\qp)\longrightarrow B_p$ be an approximation of $\iota_p^{-1}$ to $p$-adic precision $2m=v_p(\det(vu))$ and relative to the lattices $M_2(\zp)$ and $R_p$. Then $Hom_{\Gamma^1}(u,v)$ is non-empty if and only if the shortest vectors in the $\z$-lattice
		\begin{equation*}
			\Lambda(u,v)=\left(f\left(v M_2(\zp)u^*\right) + p^{2m+1}R\right)\cap R^n=f\left(v M_2(\zp)u^*\right)\cap R^n + p^{2m+1}R^n
		\end{equation*}
		have reduced norm $p^{2m}$.
	\end{lemma}
	\begin{remark}
		\begin{enumerate}[label=(\alph*)]
			\item All the elements in $\Lambda(u,v)$ have reduced norm at least $p^{2m}$.
			\item The crucial Remark 3.5 of \cite{FrancMasdeu2014} guarantees that the lattice $\Lambda(u,v)$ can be computationally described starting from a basis of $R$, under the requirement of considering an approximation of $\iota_p$ of $p$-adic precision at least $2m$.
			\item The algorithm of \cite{FrancMasdeu2014}, as developed in \cite{FrancMasdeuSagecode2011}, relies on the LLL algorithm for finding the shortest vector in $\Lambda(u,v)$, with a complexity of $O(m^3)$.
		\end{enumerate}
	\end{remark}
	
	For understanding the action of $\Gamma$ on the left-quotient $GL_2(\qp)/\qp^\times \Gamma_0(p^n\zp)$, we can combine the two equivariant isomorphisms in Equations (\ref{geodesics iso}) and (\ref{iota p iso}) and consider finite-length geodesics on the Bruhat--Tits tree. Since the algorithm of \cite{FrancMasdeu2014} deals with vertices of the Bruhat--Tits tree, we can suitably adjust it in order to compute representatives of our double-coset space of interest.
	
	\begin{algorithm}[H]\label{algorithm}
		\SetAlgoVlined
		\KwIn{The prime $p$, the quaternion algebra $B/\q$, the order $R^n$ and a fixed isomorphism $\iota_p$ as above.}
		\KwOut{A list of coset representatives for $\iota_p(\Gamma^1) \backslash GL_2(\qp) / \qp^\times \Gamma_0(p^n\zp)$.}
		\Begin{
			Initialize queue $W$ with the privileged vertex $v_0$;\\
			Initialize two empty lists $E$ and $P$;\\
			\While{$W\neq \emptyset$}{
				\textbf{Pop} $v$ from $W$;\\
				\For{each $g\in Geod_{n}(v)$}{
					\If{there is no $g'\in Geod_{n}(v)$ which is $\Gamma^1$-equivalent to $g$}{
						Append $g$ to $E$\;
						\eIf{there is a vertex $v'\in W$ which is $\Gamma^1$-equivalent to $t(g)$}{
							Append $(t(g),v',\gamma)$ to $P$, where $\gamma\in \Gamma^1$ is such that $\gamma t(g) = v'$}{
							\textbf{Push} $t(g)$ onto $W$}
					}
				}
			}
			\Return{\emph{$E$ }and \emph{$P$}.}
		}
		\caption{Compute a coset decomposition for $\Gamma^1 \backslash B_p^\times / \qp^\times {R^n_p}^\times$.}%and a fundamental domain for $\Gamma^1\backslash \mathcal{T}$.}
\end{algorithm}
\begin{remark}
	\begin{enumerate}[label=(\alph*)]
		\item The complexity of the algorithm has to take into account the number of geodesics. We already noticed that for each vertex $v$, $\# Geod_n(v)=p^{n-1}(p+1)$ and thus the algorithm has to check the $\Gamma^1$-equivalence for $p^{n-1}(p+1)$-points for each $v$, each time for each terminus point of a chosen geodesic. Furthermore, Lemma \ref{coset matrices} for $r\geq 2$ requires undoubtedly more computational resources compared to the case $r=0,1$; this is also due to the presence of elements in $\zp^\times$ and not only in $\z$.
		\item The algorithm ends in a finite number of steps, as the double quotient has finite cardinality. In fact, it is at most 2 times the class number of $R^n$, as it can be deduced by Proposition \ref{gamma gamma 1 index} below.
	\end{enumerate}
\end{remark}

\subsection{The algorithm for $\Gamma$-classes}\label{the second algorithm}

As the above algorithm computes a set of representatives for the action of the elements of reduced norm 1, it remains to compute a set of representatives for the action of $\Gamma=(R^1[1/p])^\times$ or, more precisely, for the action of $\iota_p(\Gamma)$. We recall that the reduced norm of any element in $B[1/p]$ is positive, as $B$ is definite. We begin presenting an example in which $\Gamma^1$ and $\Gamma$ are distinct.
\begin{example}\label{example element of order 1/p}
	Let $B$ be a quaternion algebra over $\q$ ramified exactly at $3$ and $\infty$ and fix the four generators to be $\langle 1, i,j,k\rangle$ such that $i^2=-1$, $j^2=-3$ and $ij=k$. Example $2$ in Section 9 of \cite{Pizer1980} provides an explicit basis for an Eichler order of level $5$, namely
	\begin{equation*}
		R^1=\langle \tfrac{1}{2}(1+j+2k), \tfrac{1}{2}(i+5k), j+2k,5k \rangle_\z
	\end{equation*}
	and it is not difficult to notice that the elements $1\pm 2i\in R^1$ have reduced norm $5$. Even though they are not invertible in $R^1$, they can be inverted in $R^1[1/5]$ with inverse (resp.) $(1\mp 2i)/5$. In particular, the reduced norm of $(1\mp 2i)/5$ is $1/5$, hence $(1\mp 2i)/5\in \left(R^1[1/5]\right)^\times-\left(R^1[1/5]\right)^1$.
\end{example}
\begin{lemma}\label{norm +-p^n elements}
	Let, if it exists, $\delta_p$ be an element in $(R^n[1/p])^\times$ of reduced norm $p$. For any $\gamma\in (R^n[1/p])^\times$ with reduced norm $n(\gamma)=p^m$, $m\in\z$, set $M=m/2$ if $m$ is even and $M=m$ otherwise. The element $\gamma\cdot \delta_p^{-M}$ belongs to $(R^n[1/p])^1$.
\end{lemma}
\begin{proof}
	The order $R^n[1/p]=R^1[1/p]$ is a $\z[1/p]$-order and an element belongs to $(R^n[1/p])^\times$ if and only if its reduced norm lies in $\z[1/p]^\times=\pm p^\z$ (see \cite{Vigneras1980}, Lemme 4.12) and moreover it is positive. If $n(\gamma)=p^{2m}$ it is obvious, otherwise, since $B$ is definite, there are only finitely many elements of given reduced norm, hence, by finite enumeration, we can take one of these with reduced norm $p$ (if any). In particular, this element of reduced norm $p$ can be taken to lie in $R^n$.
\end{proof}
\begin{remark}
	Let $x$ and $y$ be two elements in $\Gamma^1 \backslash B_p^\times / \qp^\times {R^n_p}^\times$ and suppose that they are expressed as in Lemma \ref{coset matrices2}. If they are $\Gamma$-left-equivalent, then $n(x)/n(y)\in p^{1+2\z}$.
\end{remark}
\begin{proposition}
	A set of representatives for $\Gamma\backslash Geod_n(\mathcal{T})$ can be extracted, by finite enumeration, from the a set of representatives for the quotient space $(R^n[1/p])^1\backslash Geod_n(\mathcal{T})$.
\end{proposition}
\noindent The proof of the above proposition is the following algorithm.\\

\begin{algorithm}[H]\label{algorithm-gamma}
	\SetAlgoVlined
	\KwIn{A coset decomposition provided by Algorithm \ref{algorithm} and the class number of $R^n$.}
	\KwOut{A list of coset representatives for $\iota_p(\Gamma) \backslash GL_2(\qp) / \qp^\times \Gamma_0(p^n\zp)$.}
	\Begin{
		Initialize queue $X$ with the coset representatives of $\Gamma \backslash B_p^\times / \qp^\times {R^n_p}^\times$;\\
		Initialize empty queue $Y$;\\
		Set $h$ as the class number of $R^n$;\\
		\If{$length(X)-h = 0$}{
			Set $Y=X$;
		}
		\Else{
			Compute $\delta_p$;\\
			\While{$length(X)-h\neq 0$}{
				\For{each $x\in X$}{
					Compute $n(x)$;\\
					\textbf{Push} $x$ into $Y$;\\
					\textbf{Pop} $x$ from $X$;\\
					\For{each $x'\in X$}{
						Compute $n(x')$;\\  \If{$n(x)/n(x')=p^{m}\in p^{\z}$}{
							\If{$\delta_p^{-M}\cdot x'$ is $\Gamma^1$-equivalent to $x$}{
								\textbf{Pop} $x'$ from $X$;}
						}
					}
				}
			}
		}
		\Return{$Y$}
	}
	\caption{Compute a coset decomposition for $\Gamma \backslash B_p^\times / \qp^\times {R^n_p}^\times$.}
\end{algorithm}
$\\$
We report here the following proposition.

\begin{proposition}\label{gamma gamma 1 index}
	The index $[\Gamma:\Gamma^1]$ is either 1 or 2. It holds $[\Gamma:\Gamma^1]=2$ if and only if $R^1$ contains an element of reduced norm $p$.
\end{proposition}
\begin{proof}
	Recall that $\Gamma$ and $\Gamma^1$ are respectively the images in $PGL_2(\qp)$ of $(R^1[1/p])^\times$ and $(R^1[1/p])^1$. In particular, we have the exact sequence
	\begin{equation*}
		1 \longrightarrow (R^1[1/p])^1 \longrightarrow (R^1[1/p])^\times \overset{n}{\longrightarrow} p^{\z},
	\end{equation*}
	thus every element in $\Gamma$ can be divided into two classes: the class of elements with reduced norm $p^{2m}$ and the one of elements with reduced norms $p^{2m+1}$. Defining the character $ns:\Gamma\longrightarrow \{\pm1\}$ as $ns([\gamma])=(-1)^{n(\gamma)}$, we obtain the exact sequence
	\begin{equation*}
		1 \longrightarrow \Gamma^1 \longrightarrow \Gamma \overset{ns}{\longrightarrow} \{\pm1\}.
	\end{equation*}
	Hence $ns$ is surjective if and only if there exists an element of reduced norm $p$ in $(R^1[1/p])^\times$ or, equivalently, in $R^1$.
\end{proof}
\begin{remark}\label{BD and D remark on norm p}
    \begin{enumerate}
        \item As noticed in \cite{BertoliniDarmon2007}, Section 3.1, one can say something more in the case the Eichler order contains the ring of integers of a quadratic imaginary field. Let $K$ be such a field and let $\mathcal{O}_K$ be its ring of integers. Assume that $\ell$ is inert in $K$ and that all the primes dividing $Np$ are split. This assumption guarantees the existence of an embedding of $\mathcal{O}_K$ into the Eichler order $R^1$. In particular, $R^1$ contains an element of norm $p$ if $\mathcal{O}_K$ contains one of the ideals splitting $(p)$.
        \item In \cite{Darmon2004}, Section 9.1, one can notice that the element of norm $p$ acts changing orientation. Let us remark that this is not always the case when one considers geodesics with bigger length; see matrix $\mat{0}{1}{p^n}{0}$ in \cite{Rhodes2001}, Section 2.
    \end{enumerate}
\end{remark}

\subsection{A basis for the order}\label{basis for orders}

For the above algorithms to be effective we should provide a basis of the order $R^n$. We refer here to some algorithmic procedures for computing a basis for $R^n$.

We begin considering a maximal order; in this case, we have the following lemma.
\begin{lemma}[\cite{Pizer1980}, Proposition 5.2]\label{basis}
	
	Let $B$ be as above and suppose that it is determined by $i^2=a$ and $j^2=b$; set $k=ij$. A maximal order ${R}$ of $B$ is
	\begin{enumerate}
		\item for $a=-1$, $b=-\ell$ if $\ell\equiv 3\pmod{4}$,
		\begin{equation*}
			{R}=\langle \frac{1}{2}(1+j), \frac{1}{2}(i+k), j, k\rangle_{\z},
		\end{equation*}
		\item for $a=-2$, $b=-\ell$ if $\ell\equiv 5 \pmod{8}$,
		\begin{equation*}
			{R}=\langle \frac{1}{2}(1+j+k), \frac{1}{4}(i+2j+k), j, k\rangle_{\z},
		\end{equation*}
		\item for $a=-\ell$, $b=-q$ if $\ell\equiv 1\pmod{8}$ where $q$ is a prime such that $\left(\frac{\ell}{q}\right)=-1$, $q\equiv 3\pmod{4}$ and $s$ is an integer with $s^2=-\ell \pmod{q}$ and $s\equiv -q \pmod{}\ell$,
		\begin{equation*}
			R=\langle \frac{1}{2}(1+j), \frac{1}{2}(i+k), k, \frac{1}{\ell q}\left(\ell j+(s+q)k \right)\rangle_{\z}.
		\end{equation*}
	\end{enumerate} 
\end{lemma}
\noindent In \cite{KirschmerVoight2014} (end of page 5), it is described how one can always compute a $\z$-basis for a given Eichler order in $B$. Up to conjugation, we can consider the Eichler order $R_N\subseteq R$ of level $N$ to be, locally, $\Gamma_0(N\z_q)$. Factoring $N$ into prime factors, one can compute in probabilistic polynomial time an embedding of ${R}\hookrightarrow M_2(\z_p)$ for $p|N$, as explained in \cite{Voight2013}. Once the embedding has been computed, it is possible to realize the order $R_{p^{v_p(N)}}$ and produce a $\z_p$-basis. Taking the intersection $R_N=\cap_{p|N}R_{p^{v_p(N)}}$, one computes (see \emph{loc.cit.}) a $\z$-basis for the Eichler order $R_N$. %In a similar fashion it is possible to compute a $\z$-basis for a special order $R_{N\ell^2}$ producing a basis of the intersection of a suitable Eichler order and an order as in the above Lemma \ref{basis}.
Another interesting approach is considered in \cite{Wiebe2018}, where $\z$-basis of non maximal orders are provided and computed.

We end this section remarking that one needs to only compute once the Eichler order with level $Np$ to address the $\Gamma$-structure, as the higher powers at $p$ are handled by the length of the geodesics.

%%%%%%%%%%%%%%%%%%%%%%%%%%%%%%%%%%%%%%%%%%%%%%%%%%%%%%%%%%%%%%

\section{The triple product $p$-adic $L$-function}\label{triple product section}

The aim of this section is to apply the algorithmic procedures developed in Section \ref{Quaternionic double quotients and geodesics on the Bruhat--Tits tree} to the triple product balanced $p$-adic $L$-function presented in \cite{Hsieh2021}; we try to keep the notation consistent with \emph{loc.cit.} and denote this $p$-adic $L$-function by $\mathcal{L}_{F_\infty}^{bal}$, for $F_\infty$ a suitable triple of Hida families. More precisely, in this section we realize the so-called \emph{theta-element} in \emph{loc.cit.} as a finite sum of functions on the finite-length geodesics on the Bruhat--Tits tree. This produces an effective method for the approximation of the limit of the triple product $p$-adic $L$-function at $(2,1,1)$.

We recall, as briefly as possible, the setting of \cite{Hsieh2021}. We are interested in considering weight-2 specializations in the formulas, thus we report only what is needed for this purpose, leaving out great part of the general construction; we refer to the original paper for all the details.

\subsection{The setting and the hypotheses}\label{hypotheses}

As before, let $p$ be an odd prime which, for this section, is assumed to be greater equal than 5. We consider three $p$-adic Hida families, ${f_\infty}$, ${g_\infty}$ and ${h_\infty}$, with, respectively, tame level $N_1$, $N_2$ and $N_3$, and tame character $\psi_1$, $\psi_2$ and $\psi_3$. We denote the Iwasawa algebra by $\Lambda=\zp[[\zp^\times]]$, and define $\mathbf{I}_i$, for $i=1,2,3$, to be the finite flat normal domain over $\Lambda$ which defines the coefficients algebra of the corresponding family.
As in Section \ref{weight space}, an arithmetic point in the weight space is induced by a morphism $Q\in \textrm{Spec}(\mathbf{I}_i)(\overline{\qp})$ such that the restriction $Q_{|_{\zp^\times}}:{\zp^\times}\rightarrow \Lambda^\times \xrightarrow[]{Q} \overline{\qp}^\times$ is of the form $Q(x)=x^{k_Q}\cdot \varepsilon_Q(x)$ for $k_Q\in \z_{\geq 2}$ and $\varepsilon_Q:{\zp^\times}\longrightarrow \overline{\qp}^\times$ a finite order character. We denote the exponent of its conductor by $c\left(\varepsilon_Q\right)$. For $N=\lcm\{N_1,N_2,N_3\}$ and for any $\underline{Q}=(Q_1,Q_2,Q_3)$ triple of arithmetic points, we denote
\begin{equation*}
	\Sigma^-=\left\{q \textrm{ prime dividing } N\mid\varepsilon_q\left( \Pi_{\underline{Q}} \otimes\chi_{\underline{Q}}^{-1} 	\right)=-1\right\},
\end{equation*}
where $\varepsilon_q$ is the local epsilon factor at the prime $q$. Here, $\Pi_{\underline{Q}}=\pi_{{f}_{Q_1}}\otimes\pi_{{g}_{Q_2}}\otimes\pi_{{h}_{Q_3}}$ is the tensor product of the three automorphic representations associated with the specializations ${f}_{Q_1}$, ${g}_{Q_2}$ and ${h}_{Q_3}$ and $\chi_{\underline{Q}}=\prod_{i=1}^{3}\psi_i\varepsilon_{Q_i}$. We remark that $\Sigma^-$ does not depend on the specific triple of arithmetic points.

We recall briefly all the hypotheses that the Hida families must satisfy:
\begin{enumerate}
	\item[\hla{(sf)}] $\gcd(N_1,N_2,N_3)$ is square-free;	
	\item[\hla{(ev)}]	$\psi_1\psi_2\psi_3=\omega^{2a}$ for $a\in\z$ and $\omega$ the Teichm\"uller character;
	\item[\hla{(CR,$\Sigma^-$)}]	The residual Galois representation $\overline{\rho}_{{f_\infty}}:\mathsf{Gal}(\overline{\q}/\q)\longrightarrow GL_2(\overline{\f}_p)$ (and the same holds for ${g_\infty}$ and ${h_\infty}$) is such that
	\begin{enumerate}
		\item[(i)]	$\overline{\rho}_{{f_\infty}}$ is absolutely irreducible;
		\item[(ii)]	$\overline{\rho}_{{f_\infty}}$ is $p$-distinguished;
		\item[(iii)]	if $\ell\in\Sigma^{-}$ and $\ell\equiv 1 \pmod{p}$, then $\overline{\rho}_{{f_\infty}}$ is ramified at $\ell$;
	\end{enumerate}
	\item[\hla{(odd)}]  $\#\Sigma^-$ is odd.
\end{enumerate}
One crucial hypothesis is the hypothesis of \emph{tame ramification} at the primes of $\Sigma^{-}$, that is
\begin{enumerate}
	\item[\hla{(TR)}] for $d=\prod_{q\in \Sigma^-}q$ and $N=\lcm(N_1,N_2,N_3)$, it must hold that $\left(d,N/d\right)=1$.
\end{enumerate}
In order to explicitly produce test vectors, in \cite{Hsieh2021} it is required to introduce a permutation of the triple of families. We report the two equivalent (see Remark 3.1 in \emph{loc.cit.}) conditions:

\begin{enumerate}
	\item[\hla{(P)}]	There exists a triple of classical points $(\kappa_1,\kappa_2,\kappa_3)$ such that for each prime $q|N$, there exists a rearrangement $\{f_1,f_2,f_3\}$ of $\{f_{\kappa_1},g_{\kappa_2},h_{\kappa_3}\}$ such that
	\begin{enumerate}[label=(\roman*),font=\itshape]
		\item	the conductors at $q$ satisfy $c_q(\pi_{f_1})\leq \min\{c_q(\pi_{f_2}),c_q(\pi_{f_3})\}$;
		\item	the local components $\pi_{f_1,q}$ and $\pi_{f_3,q}$ are minimal;
		\item	either $\pi_{f_3,q}$ is a principal series or $\pi_{f_2,q}$ and $\pi_{f_3,q}$ are both discrete series.
	\end{enumerate}
	\item[\hla{(P$\infty$)}]	For each triple of classical points $(\kappa_1,\kappa_2,\kappa_3)$ and for each prime $q|N$, there exists a rearrangement $\{f_1,f_2,f_3\}$ of $\{f_{\kappa_1},g_{\kappa_2},h_{\kappa_3}\}$ such that
	\begin{enumerate}[label=(\roman*),font=\itshape]
		\item	$c_q(\pi_{f_1})\leq \min\{c_q(\pi_{f_2}),c_q(\pi_{f_3})\}$;
		\item	the local components $\pi_{f_1,q}$ and $\pi_{f_3,q}$ are minimal;
		\item either $\pi_{f_3,q}$ is a principal series or $\pi_{f_2,q}$ and $\pi_{f_3,q}$ are both discrete series.
	\end{enumerate}
\end{enumerate}
For completeness, we report here the last hypothesis needed to compute local zeta-integrals at $p$.
\begin{enumerate}
	\item[\hla{(Hb')}] Either $\pi_{f_3,p}$ is a principal series or all $\pi_{f_1,p}$, $\pi_{f_2,p}$, and $\pi_{f_3,p}$ are special.
\end{enumerate}

\begin{remark}
	As stated in Remark 6.2 of \cite{Hsieh2021}, there always exist $\chi_1$, $\chi_2$ and $\chi_3$ Dirichlet characters modulo $M$, with $M^2|N$ such that $\chi_1\chi_2\chi_3=1$ and $(\pi_{f_{\kappa_1}}\otimes\chi_1,\pi_{g_{\kappa_2}}\otimes\chi_2,\pi_{h_{\kappa_3}}\otimes\chi_3)$ satisfies \emph{\hrf{(P)}}. It is not hard to note that this condition fails if we relax condition \emph{\hrf{(TR)}} and allow powers strictly higher than 2 at the primes of $\Sigma^{-}$.
\end{remark}

\subsection{Test vectors and the theta-elements}

Let $B$, $R$ and $R^n$ be as in Section \ref{p-adic double-coset spaces}. Moreover, we fix a choice of field embeddings $\overline{\q}\hookrightarrow\overline{\qp}\hookrightarrow \cc$ and a compatible isomorphism $i:\cc_p\cong \cc$.  Let $\left(\begin{smallmatrix}
	a & b\\ c& d
\end{smallmatrix}\right)\in GL_2(\A_{\q,f}^{(\ell)})$ act on $x\in\widehat{B}^\times$ by
\begin{equation*}
	x\left(\begin{smallmatrix}
		a & b\\ c& d
	\end{smallmatrix}\right)=x\cdot \iota^{-1}\left(\left(\begin{smallmatrix}
		a & b\\ c& d
	\end{smallmatrix}\right)\right),
\end{equation*}
for $\iota$ the induced isomorphism $GL_2(\A_{\q,f}^{(\ell)})\cong (B\otimes\A_{\q,f}^{(\ell)})^\times$.
Even though the formula presented in \cite{Hsieh2021} can deal with generic arithmetic specializations, we restrict our analysis to the one of weight 2. This situation grants easier expressions, without invalidating the study of the limit point $(2,1,1)$.
\begin{assumption}[Choice of Hida families]\label{choice of Hida families}
	We consider the triple of families ${F_\infty}=({g_\infty},{f_\infty},{h_\infty})$, where ${f_\infty}$ is passing through a modular elliptic curve. We aim to let the family ${f_\infty}$ to be as free as possible, thus the families ${g_\infty}$ and ${h_\infty}$ should be interpreted as auxiliary parameters. More precisely we take:
	\begin{enumerate}[label=(\alph*)]
		\item ${f_\infty}$ is the unique Hida family\commenttt{(see \cite{Wiles1988},\cite{hida_1993})} associated with $f\in S_2^{new}(\Gamma_0(N_1\ell p),\q)$, a twist-minimal primitive newform corresponding to an elliptic curve $E/\q$, which is ordinary at $p$. In particular, the family has tame level $N_1\ell$ with trivial tame character;
		\item ${g_\infty}$ and ${h_\infty}$ are primitive $p$-adic Hida families of tame level $N_2\ell$ with tame character $\psi$ and $\psi^{-1}$ respectively. We moreover suppose that $\psi$ and $\psi^{-1}$ are both primitive of conductor $N_2$.
	\end{enumerate}
	In the end, we assume that $N_2$ is square-free, $N_2|N_1$ and that all the three families satisfy condition \emph{\hrf{(CR,$\Sigma^-$)}}.
\end{assumption}
\begin{remark}
	\begin{enumerate}[label=(\alph*)]
		\item The permutation chosen in the definition of $F_\infty$ guarantees that the triple satisfies \emph{\hrf{(P)}}, without any auxiliary twist. 
		\item One can also consider $\psi$ of smaller conductor $N_2'\mid N_2$ and impose the condition on $f$ to have supercuspidal type at the primes dividing $N_2'$. This ensures that the set of ramified primes contains just $\ell$ (see \cite{P1990}).
		\item The hypothesis \emph{\hrf{(Hb')}} is in place under our choice of specializations.
	\end{enumerate}
\end{remark}
Let $F_\infty^{B}=(g_\infty^B,f_\infty^B,h_\infty^B)$ be the triple of \emph{primitive Jacquet--Langlands lifts} as determined in Theorem 4.5 of \cite{Hsieh2021}. 
\begin{notation}\label{Eichler notation}
	We fix the orders $R^n$ and $R_2^n$ to be nested Eichler orders of level, respectively, $N_1\ell p^n$ and $N_2\ell p^n$ such that $R\supset R^n_2\supset R^n$ for each $n\geq 0$.
\end{notation}
We fix a sequence of points $(2,\varepsilon)$ on the weight space, as in Section \ref{weight space}, approaching a weight-1 point with trivial character; we assume that $\varepsilon$ is primitive of conductor $p^n$. Considering the notation in Section \ref{Quaternionic modular forms of weight $2$}, at each point of the form $((2,\varepsilon),2,(2,\varepsilon))$ we have the quaternionic test vectors
\begin{equation*}
	\left(i(g^B_{(2,\varepsilon)})\otimes\varepsilon^{-1}_{\A},i(f^B_2),i(h^B_{(2,\varepsilon)})\right)\in S_2(R^n,\widetilde{\psi\varepsilon^{-1}})\times S_2(R^n)\times S_2(R^n,\widetilde{\psi^{-1}\varepsilon}),
\end{equation*}
where $\varepsilon^{-1}_{\A}$ is the ad\`{e}lization of the $p$-adic character $\varepsilon^{-1}$, which is unramified outside $p$. We set ${g^B_{(2,\varepsilon)}}'=g^B_{(2,\varepsilon)}\otimes\varepsilon^{-1}_{\A}$ and ${F_\infty^{B}}'=({g_\infty^B}'=g_\infty^B\otimes \varepsilon_\A^{-1},f_\infty^B,h_\infty^B)$. Following Section 4.6 of \cite{Hsieh2021}, we consider the theta-element $\Theta_{{F_\infty^{B}}'}$ associated with the triple ${F_\infty^{B}}'$.
\begin{remark}\label{theta e L}
	\begin{enumerate}[label=(\alph*)]
		\item The specialization point $Q=(2,\varepsilon)$ comes from an element in $\textrm{Spec}(\mathbf{I}_i)(\overline{\qp})$ and hence, \emph{a priori}, we should take into account the value of $\varepsilon_{\A}(p)=Q(p)$, although we can fix it to be $1$.
		\item The theta-element $\Theta_{{F_\infty^{B}}'}$ is the central component of the $p$-adic triple product $L$-function $\mathcal{L}_{F_\infty}^{bal}$ constructed in Section 4 of \cite{Hsieh2021}. In our situation, the difference between these two objects is give by two factors: the constant $2^{-5/2}/\sqrt{N}$ and an element in $Frac\left(\widehat{\bigotimes}_{i=1,2,3}\mathbf{I}_i\right)^\times$ (in the notation of  \emph{loc.cit.} a square root of the fudge factor). In particular, the zeros of $\mathcal{L}_{F_\infty}^{bal}$ are encoded by the behavior of $\Theta_{{F_\infty^{B}}'}$.
	\end{enumerate}
\end{remark}
By some elementary but rather tedious observations, it is not difficult to note that, up to considering the finite sum rather than the integral, Proposition $4.9$ in \cite{Hsieh2021} reads as
\begin{multline*}
	\Theta_{{F_\infty^{B}}'}(((2,\varepsilon),2,(2,\varepsilon)))=\frac{\varepsilon^{-1}_{\A}\left(p^{n}\right)\cdot p^{2n}(1-\frac{1}{p})}{a_p(f)^n a_p\left({g^B_{(2,\varepsilon)}}'\right)^n a_p\left({h^B_{(2,\varepsilon)}}\right)^n}\cdot\\ 
	\cdot \sum_{x\in B^\times \backslash \widehat{B}^\times / \widehat{R^n}^\times } \frac{1}{\#\Gamma_{x}} i\left({g^B_{(2,\varepsilon)}}'\left(x\left(\begin{smallmatrix}
		1 & p^{-n}\\
		0 & 1
	\end{smallmatrix}\right)\right)\right)\cdot i\left(f^{B}\left(x\right)\right)
	i\left({h^B_{(2,\varepsilon)}}\left(x\left(\begin{smallmatrix}
		0 & p^{-n}\\
		-p^{n} & 0
	\end{smallmatrix}\right)\right)\right).	
\end{multline*}
In the above formula we recall that $\Gamma_{x}=\left(B^\times \cap x\widehat{R^n}^\times x^{-1}\right)\q^\times/\q^\times$.\\

We can move the quaternionic modular forms to their exact level considering the two surjective reduction maps
\begin{align*}
	\phi_{n}:B^\times \backslash \widehat{B}^\times / \widehat{R^n}^\times\longrightarrow B^\times \backslash \widehat{B}^\times / \widehat{R^n_2}^\times && \textrm{ and } && \phi^1_{n}:B^\times \backslash \widehat{B}^\times / \widehat{R^n}^\times\longrightarrow B^\times \backslash \widehat{B}^\times / \widehat{R^1}^\times.
\end{align*}		
Lemma \ref{p-adic model} allows us to identify the triple of functions
\begin{equation*}
	\left(i\left({g^B_{(2,\varepsilon)}}'\left(\phi_{n}\left(-\right)\right)\right),i\left(f^{B}\left(\phi^1_{n}(-)\right)\right), i\left({h^B_{(2,\varepsilon)}}\left(\phi_{n}\left(-\right)\right)\right) \right),
\end{equation*}
with a triple of functions on $GL_2(\qp)$ which we denote by $\left(G^n(-),F^1(-),H^n(-)\right).$
\begin{proposition}\label{Theta element}
	With the notation of Section \ref{geodesics}, let $\underline{Q}$ be the triple of points $((2,\varepsilon),2,(2,\varepsilon))$. It holds that
	\begin{multline*}
		\Theta_{{F_\infty^{B}}'}(\underline{Q})=\frac{(1-p^{-1})}{a_p(f)^n  }\cdot \sum_{e\in \left(R^1\left[1/p\right]\right)^\times\backslash \mathcal{E}(\mathcal{T})} \frac{F^1(e)}{\# Stab_{\left(R^1\left[1/p\right]\right)^\times}(e)}\cdot\\ 
		\cdot \left(\frac{\varepsilon^{-1}_{\A}\left(p^{n}\right)\cdot p^{2n}}{a_p\left({g^B_{(2,\varepsilon)}}'\right)^n a_p\left({h^B_{(2,\varepsilon)}}\right)^n}\right)
		\sum_{g \in Geod_n(\mathcal{T})(e)} G^n(g\left(\begin{smallmatrix}
			1 & p^{-n}\\
			0 & 1
		\end{smallmatrix}\right))\cdot H^n(g\left(\begin{smallmatrix}
			0 & p^{-n}\\
			-p^{n} & 0
		\end{smallmatrix}\right)).
	\end{multline*}
\end{proposition}
\begin{proof}
	We begin noting that $\Gamma_{x}=(B^\times \cap x \widehat{R^n}^\times x^{-1})\q^\times/\q^\times$ corresponds, by Lemma \ref{p-adic model}, to
	\begin{equation*}
		\Gamma_{p,x_p}=\left(\left({R^n}\left[1/p\right]\right)^\times\cap x_p \Gamma_0(p^{n}\zp) x_p^{-1}\right)\q^\times/\q^\times=\left(\left({R^n}\left[1/p\right]\right)^\times\cap x_p \Gamma_0(p^{n}\zp) x_p^{-1}\right)\qp^\times/\qp^\times.
	\end{equation*}
	The theta-element becomes
	\begin{multline*}
		\Theta_{{F_\infty^{B}}'}(\underline{Q})=\frac{\varepsilon^{-1}\left(p^{n}\right)\cdot p^{2n}(1-p^{-1})}{a_p(f)^{n} a_p\left({g^B_{(2,\varepsilon)}}'\right)^{n} a_p\left(h^{B}_{(2,\varepsilon)}\right)^{n}}\cdot\\ 
		\cdot \sum_{x\in\iota_p\left({R^n}\left[1/p\right]^\times\right)\backslash GL_2(\qp) / \qp^\times \Gamma_0(p^{n}\zp )} \frac{1}{\#\Gamma_{p,x_p}}  F^1(x_p)\cdot  G^n(x_p\left(\begin{smallmatrix}
			1 & p^{-n}\\
			0 & 1
		\end{smallmatrix}\right))\cdot H^n(x_p\left(\begin{smallmatrix}
			0 & p^{-n}\\
			-p^{n} & 0
		\end{smallmatrix}\right)),
	\end{multline*}
	and thus, under the identification with geodesics on the Bruhat--Tits tree,
	\begin{multline*}
		\Theta_{{F_\infty^{B}}'}(\underline{Q})=\frac{\varepsilon^{-1}\left(p^{n}\right)\cdot p^{2n}(1-p^{-1})}{a_p(f)^{n} a_p\left({g^B_{(2,\varepsilon)}}'\right)^{n} a_p\left(h^{B}_{(2,\varepsilon)}\right)^{n}}\cdot\\ 
		\cdot \sum_{g\in \iota_p\left({R^n}\left[1/p\right]^\times\right)\backslash Geod_n(\mathcal{T})} \frac{1}{\# Stab_{\left({R^n}\left[1/p\right]\right)^\times}(g)}  F^1(g)\cdot  G^n(g\left(\begin{smallmatrix}
			1 & p^{-n}\\
			0 & 1
		\end{smallmatrix}\right))\cdot H^n(g\left(\begin{smallmatrix}
			0 & p^{-n}\\
			-p^{n} & 0
		\end{smallmatrix}\right)),
	\end{multline*}
	where $Stab_{\left({R^n}\left[1/p\right]\right)^\times}(g)$ corresponds to $\Gamma_{p,x_p}$. In the above equation we have identified the triple $(F^1(-),G^n(-),H^n(-))$ as a triple of functions on the geodesic of the Bruhat--Tits tree as in Section \ref{p-adic model}. In order to ease the notation, we drop again the isomorphism $\iota_p$ and write $\left({R^n}\left[1/p\right]\right)^\times$ for the corresponding subgroup in $GL_2(\qp)$. Since the function $F^1$ is naturally a function on the edges of the Bruhat--Tits tree (it can indeed be identified with the harmonic cocycle associated with $f$) we can proceed similarly to equation $(2.2)$ of \cite{Rhodes2001} and applying Lemma \ref{left inv order} obtain
	\begin{multline*}
		\Theta_{{F_\infty^{B}}'}(\underline{Q})=\frac{\varepsilon^{-1}\left(p^{n}\right)\cdot p^{2n}(1-p^{-1})}{a_p(f)^{n} a_p\left({g^B_{(2,\varepsilon)}}'\right)^{n} a_p\left(h^{B}_{(2,\varepsilon)}\right)^{n}}\cdot\\ 
		\cdot \sum_{e\in \left({R^n}\left[1/p\right]\right)^\times\backslash \mathcal{E}(\mathcal{T})} \frac{F^1(e)}{\# Stab_{\left({R^n}\left[1/p\right]\right)^\times}(e)}  
		\sum_{g \in Geod_n(\mathcal{T})(e)} G^n(g\left(\begin{smallmatrix}
			1 & p^{-n}\\
			0 & 1
		\end{smallmatrix}\right))\cdot H^n(g\left(\begin{smallmatrix}
			0 & p^{-n}\\
			-p^{n} & 0
		\end{smallmatrix}\right)).
	\end{multline*}
\end{proof}
We are now ready to apply the algorithms developed in Sections \ref{the algorithm} and \ref{the second algorithm}, and obtain the following theorem.
\begin{theorem}\label{theorem approximation L}
	The value at $(2,1,1)$ of the balanced $p$-adic $L$-function associated with the triple of families $F_\infty$  is, up to an element in $Frac\left(\widehat{\bigotimes}_{i=1,2,3}\mathbf{I}_i\right)^\times$,
	\begin{align*}
		\mathcal{L}^{bal}_{F_\infty}(2,1,1)\overset{\cdot}{\equiv} \lim\limits_{\varepsilon\rightarrow 1}  \Theta_{{F_\infty^{B}}'}((2,\varepsilon),2,(2,\varepsilon)).
	\end{align*}
	The limit $\lim\limits_{\varepsilon\rightarrow 1} \Theta_{{F_\infty^{B}}'}((2,\varepsilon),2,(2,\varepsilon))$ can be algorithmically approximated with a given $p$-adic precision, \emph{i.e.} for any $\varepsilon$ of finite conductor we can compute $\Theta_{{F_\infty^{B}}'}((2,\varepsilon),2,(2,\varepsilon))$ with a given $p$-adic precision.
\end{theorem}
\begin{remark}
	We can give a precise estimation of the convergence rate for the limit in Theorem \ref{theorem approximation L}, $\lim\limits_{\varepsilon\rightarrow 1}  \mathcal{L}^{bal}_{{F_\infty^{B}}}((2,\varepsilon),2,(2,\varepsilon))$. This comes from a straightforward application of one implication of Proposition 4.0.6 (\emph{abstract Kummer congruences}) in \cite{Katz1978}. We notice that one must compute explicitly the \emph{fudge factor} of Remark \ref{theta e L}.
\end{remark}

\subsection{Explicit computation of forms on the geodesics of the Bruhat--Tits tree}\label{explicit (fgh)}

    We report here a series of routines to compute the triple $(G^n,F^1,H^n)$ as a triple of functions on the geodesic of the Bruhat--Tits tree, as in Section \ref{p-adic double-coset spaces}. We wish to note that, \emph{a priori}, Theorem 4.2.(2) of \cite{Hsieh2021} is imposing a normalization and this is necessary for having a compatible choice of specializations. However, we remark that the choice of normalization does not invalidate the arithmetic behavior of the limit.\\

    We begin with the following procedure, which translates a quaternionic modular form to a form on finite-length geodesics. 
    \begin{routine}\label{tildeF}\phantom{0}
        \begin{enumerate}[label ={\emph{\textbf{Step \arabic*:}}},align=left]
            \item Apply Algorithmic routine \ref{alg routine quaternionic mod forms} to compute the sought-for quaternionic modular form, say $\varphi$. In the computation, keep track of the ideals representing the ideal classes.
        
        	\item Apply Algorithm \ref{algorithm-gamma} to produce a set of representatives for the corresponding level structure.
        
            \item Consider the approximated embedding in the algorithms of \cite{FrancMasdeu2014} to pull the matrix representative of \textbf{Step 2} back to elements of the quaternion algebra. Test the ideals computed in \textbf{Step 1} to find in which ideal each one of these representatives lies (up to a scalar in $B^\times$).
            
            \item If \textbf{Step 3} does not recover a 1:1 correspondence, increase the working precision of the embedding and go back to \textbf{Step 3}. 
        \end{enumerate}
    \end{routine}
    
    To compute the triple $(G^n,F^1,H^n)$, we must address the projections from higher to lower level. This procedure relies on Lemma \ref{proj maps} and its $PGL_2(\qp)$-equivariance as well as the fact that we are working under a common level structure given by the Eichler order $R^n$. Let $\tilde{F}$ be the from obtained the above routine. To determine the values of $\tilde{F}$ on the representatives of a smaller Eichler order, we apply the following two routines; the first one if we are varying the level at $p$, while the second one for level away from $p$. Consider Notation \ref{Eichler notation}.\\
    
    The following procedure computes the values of a form $\tilde{F}$, of level $R^m$, on the representatives for $R^n$, with $n>m$.
    \begin{routine}\label{red pn}\phantom{0}
        \begin{enumerate}[label ={\emph{\textbf{Step \arabic*:}}},align=left]
            \item Apply Algorithm \ref{algorithm-gamma} to produce a set of representatives for $R^n$.
            \item For any of these representatives, compute its class as a length-$m$ geodesic with the algorithm of Lemma \ref{coset matrices2}.
            \item Keep track of the reductions and match the corresponding values of $\tilde{F}$.
        \end{enumerate}
    \end{routine}
    
    The following procedure computes the values of a form $\tilde{F}$, of level $R_2^n$, on the representatives for $R^n$.
    \begin{routine}\label{red N}\phantom{0}
        \begin{enumerate}[label ={\emph{\textbf{Step \arabic*:}}},align=left]
            \item Apply Algorithm \ref{algorithm-gamma} to produce a set of representatives for $R^n$.
            \item For any of these representatives, compute its class under $R^n$ applying Algorithm \ref{algorithm-gamma}.
            \item Keep track of the reductions and match the corresponding values of $\tilde{F}$.
        \end{enumerate}
    \end{routine}
    
    Reiterating the above routines we obtain the forms $(G^n,F^1,H^n)$ from the triple $(\tilde{G}^n,\tilde{F}^1,\tilde{H}^n)$.
    
    \begin{remark}\label{geod and twist}
        \begin{enumerate}
            \item In Proposition \ref{Theta element}, we should consider the twist by the character $\varepsilon$. Since the Jacquet--Langlands correspondence is compatible with twists by Dirichlet characters, one can proceed as above, but running Algorithmic routine \ref{alg routine quaternionic mod forms} with the twist of a classical modular form and computing a corresponding quaternionic modular form.

            \item It remains to describe a simple task, namely the enumeration of representatives in $Geod_n(\mathcal{T})(e)$, for any class determined by an edge $e$ under $R^1[1/p]^\times$. This is done similarly to the computation of $F^1$ from $\tilde{F}^1$. More precisely, we must check which representative of $Geod_n(\mathcal{T})$ under $R^1[1/p]^\times$ is equivalent, as a matrix representing an edge, to the representative of $e$ in $R^1[1/p]^\times\backslash\mathcal{E}(\mathcal{T})$.
        \end{enumerate}
     \end{remark}
    
\subsection{A few final remarks and some motivation}

\subsubsection{A note about effective computability}\label{effective computability}

    We report here the steps involved for an effective approximation of the limit value $\lim \Theta_{{F_\infty^{B}}'}(\underline{Q})$. The algorithmic routine is as follows.
    
    \begin{routine}\label{full routine}\phantom{0}
        \begin{enumerate}[label ={\emph{\textbf{Step \arabic*:}}},align=left]
        	\item Consider a $p$-adic character of finite conductor and compute its values, together with its distance from the weight-1 character. We have hence the $p$-adic precision we are working with.
        	\item Compute a basis of the needed quaternionic orders, as in Section \ref{basis for orders}.
        	\item Run the algorithms in Sections \ref{p-adic double-coset spaces}, \ref{the algorithm} and \ref{the second algorithm}.
        	\item Fix suitable Jacquet--Langlands lifts for the Hida families and compute the values of the quaternionic specializations applying Algorithmic routine \ref{alg routine quaternionic mod forms}.
        	\item Apply Algorithmic routines \ref{tildeF}, \ref{red pn} or \ref{red N} to compute the values of the triple $(G^n,F^1,H^n)$.
        	\item Compute the action given by the matrices $\mat{1}{p^{-n}}{0}{1}$ and $\mat{0}{p^{-n}}{-p^n}{0}$ on the set of representatives obtained in the above \textbf{Step 3}.
        	\item Match the values obtained in \textbf{Step 5} with the shifts in \textbf{Step 6}.
        	\item Compute the cardinality of the stabilizers as well as the power of the eigenvalues.
        	\item Compute the finite sum in Proposition \ref{Theta element}.
        \end{enumerate}
    \end{routine}

\subsubsection{A note about effective implementation}\label{effective implementation}

    The algorithmic procedure delineated in this note is not yet supported by a reasonable pool of calculations, a lack we wish to address in the future. Unfortunately, at the present moment, there are two technical obstructions that prevent a systematic implementation.
    \begin{enumerate}
        
        \item The implementation of Brandt matrices with characters, whose main difficulty is represented by the lift of the character.
        
        \item The developed procedures require a remarkable amount of computational time and resources, as the complexity grows exponentially in the length of the geodesics, hence on the precision required.
        
    \end{enumerate}
    
\subsubsection{Some motivation for the study of the limit value at $(2,1,1)$}\label{meaning  and motivation}

	We would like to add to this note some final theoretical remarks on the arithmetic importance of the limit value at $(2,1,1)$. Let $(f,g,h)$ be a triple of modular forms as in Assumption \ref{choice of Hida families}. We can further assume that the elliptic curve associated with $f$ has analytic rank (greater) equal to 1. The Birch--Swinnerton-Dyer conjecture predicts that the algebraic rank of $E$ would be equal to its analytic rank. As above, we take the three Hida families $(f_\infty, g_\infty, h_\infty)$ passing through $(f,g,h)$. Suppose (see the next Section \ref{note on weight 1}) that $g_\infty$ and $h_\infty$ specialize to classical weight-1 modular forms; in other words, we suppose that the weight-1 specializations are associated with two Artin representations, $\rho_1$ and $\rho_2$. The value $\mathcal{L}_{(f_\infty, g_\infty, h_\infty)}^{bal}(2,1,1)$ is no more directly linked with the central value of the complex $L$-function, as the point $(2,1,1)$ lies outside the (balanced) region of interpolation. However, one expects a relation between this value and the Galois representation associated with the triple $(E,\rho_1,\rho_2)$, in a $\rho_1\otimes\rho_2$-equivariant instance of the $p$-adic BSD conjecture.
	
	The literature is rich in articles studying the values of $p$-adic $L$-functions outside the interpolation region, most notably \cite{BertoliniDarmonPrasanna2012}, where the value at $1$ of the BDP $p$-adic $L$-function for an elliptic curve is related to the $p$-adic logarithm of a Heegner point. The main tool applied in \emph{op.cit.} is the $p$-adic Gross--Zagier formula proved in \cite{BertoliniDarmonPrasanna2013}, relating images of (generalized Heegner) cycles via a $p$-adic Abel--Jacobi map to the value of the $p$-adic $L$-function. This approach revealed itself to be useful also in the case of the \emph{unbalanced} triple product $p$-adic $L$-function and its values in points outside the interpolation region. Some of the main works are \cite{DarmonRotger2014}, \cite{DarmonLauderRotger2015}, \cite{DarmonRotger2017} and the recent collective volume \cite{BDRSV_Asterisque_2022}. In \emph{op.cit.}, the geometric connotation is predominant, as the $p$-adic $L$-function is related to an indefinite quaternion algebra. In such a situation, one can study diagonal cycles on a triple product of modular curves (Shimura varieties) and associate their image via a $p$-adic Abel--Jacobi map to the special value of the $p$-adic $L$-function in a point outside the interpolation region or, more generally, to a suitable (crystalline) cohomology class, which might be interpreted as the $p$-adic avatar of the derivative of the complex $L$-function. The situation treated in this note is the opposite, namely one considers the balanced $p$-adic $L$-function which is tied to a definite quaternion algebra. Here, the geometry is quite poor as the Shimura variety is zero-dimensional. A few remarks must be made.
	\begin{enumerate}
	    \item The core of the construction of the $p$-adic $L$-function in \cite{Hsieh2021} are the \emph{regularized diagonal cycles} of its Section 4.6. These are cycles in the product of the three zero-dimensional Shimura varieties, but there is not a straightforward $p$-adic Abel--Jacobi map.
	    \item As already noticed in Remark 1.1 of \cite{Hsieh2021}, one can speculate about a relation between a $p$-adic Abel--Jacobi image of a cycle in a Shimura variety associated with a quaternion algebra ramified at $\ell$ and $p$, as in \cite{BertoliniDarmon2007}.
	    \item In the interpolation region of the balanced triple product $p$-adic $L$-function, one should not expect a meaningful relation between diagonal cycles over $\q$ and the triple product $L$-function as the latter is not necessarily zero. However, one can still construct diagonal cycles on the triple product of modular curves (or suitable indefinite Shimura varieties). Such an example can be found in Lemma 4.1 of \emph{``p-adic families of diagonal cycles''} in \cite{BDRSV_Asterisque_2022}. These peculiar cycles are defined over $\q\left(\sqrt{\left(\tfrac{-1}{p}\right)p}\right)$ and have been rather extensively studied in Sections 4.2.2 and 4.3.3 of \cite{LilienfeldtPhD2021}, where the author investigates their difference and propose a refined Beilinson--Bloch conjecture. It would be interesting to understand if there is a relation between these cycles and the zero-dimensional cycles in \cite{Hsieh2021}.
	    \item 	In a recently announced work, Andreatta--Bertolini--Seveso--Venerucci construct a new $p$-adic $L$-function considering Yoshida endoscopic lifts to $GSp_4$. They are able to study diagonal cycles on the product of Shimura varieties for $GSp_4\times GL_2\times GL_2$. This allows them to obtain a geometric interpretation of their $p$-adic $L$-function at the limit point $(2,1,1)$ as a $p$-adic limit of classes arising in the cohomology of the Siegel threefold attached to $GSp_4$.
	\end{enumerate}
	Having the above framework and motivation in mind, the advantages of an effective computation of the limit value become clear. We can test, to a given $p$-adic precision, whether the limit is zero or not (possibly providing a criterion), which must be an indicator of the existence of a non-trivial class. If the limit is non-zero, one could hope to obtain numerical clues for the behavior of the limit, similarly to \cite{DarmonLauderRotger2015}.
	
\subsubsection{A note on weight-1 modular forms and generalizations}\label{note on weight 1}
	
	The presence of classical modular forms of weight 1 (and not only $p$-adic) in a Hida family is determined by the behavior of the local automorphic representations associated with the specializations.
	\begin{lemma}[\cite{DG2012}, Lemma 4.4, or \cite{Dimitrov2014}, Proposition 1.8]\label{classical wt 1}
		Let ${f_\infty}$ be a Hida family of tame level $N$ and tame character $\psi$. If ${f_\infty}$ is of special type at some $\ell$ dividing $N$ (that is, if and only if $\ell||N$ and $\psi$ is trivial at $\ell$), then ${f_\infty}$ has no classical weight-1 specializations.
	\end{lemma}
	The triple product $p$-adic $L$-function constructed in \cite{Hsieh2021} deals only with Eichler orders, as the main working assumption is, with the above notation, the following.
	\begin{assumption}\label{hyp}
		For $M=\lcm(N_1,N_2)$, assume $\gcd(\ell,M/\ell)=1$.	
	\end{assumption}
	\noindent Assumption \ref{hyp} implies that the only local automorphic representations at $\ell$ transferring, via Jacquet--Langlands, to the quaternion algebra $B$ are the ones of special type.
	\begin{remark}
	    \begin{enumerate}
	        \item We deduce that our chosen ${g_\infty}$ and ${h_\infty}$ cannot contain classical weight-1 specializations. We must choose families with supercuspidal type at $\ell$ if we want to take Hida families passing through classical weight-1 modular forms. We refer to Section 3 of \cite{DallAva2021PhD} for a more verbose treatment and point out that one should also take into account the careful analysis in \cite{P1990}.
	        \item In \cite{GreenbergSeveso2020}, the authors provide a more general formalism for the construction of triple product $p$-adic $L$-functions, however they precisely solve the interpolation problem only in the case of Eichler orders.
	        \item The impossibility to reach classical weight-1 modular forms using the formula of Hsieh motivates the development of a rank-2 Hida theory for quaternionic lifts as in \cite{DallAva2022Hida}.
	    \end{enumerate}
	\end{remark}
	\noindent Even though these observations appear discouraging, they hide interesting arithmetic aspects. More precisely, one can still apply the algorithm to understand the limit value, even if it describes some non-standard object. Moreover, the algorithms devised in this note can be easily generalized to deal with more general quaternionic orders, as they only rely on a sufficiently explicit formulation of the balanced triple product $p$-adic $L$-function. In fact, \cite{DallAva2022Hida} is just the starting point to solve a more general interpolation problem, which we wish to address carefully in future works. Such an interpolation problem would fit perfectly in the framework announced by Andreatta--Bertolini--Seveso--Venerucci (see Section \ref{meaning  and motivation}), and it could be applied to provide computational support in developing precise arithmetic conjectures.

\appendix

\section{An illustrative example}
	
	We present an example for clarifying the algorithmic procedure developed. We do not aim to produce a meaningful computation of the limit point $(2,1,1)$ both for theoretical and technical reasons (see Section \ref{effective implementation}). We are however considering an interesting situation of \emph{exceptional zero phenomenon} for the balanced triple product $p$-adic $L$-function (see Remark 1.1. in \cite{Hsieh2021}). 
	
	A great part of the data in this example is either computed via Franc and Masdeu's code \cite{FrancMasdeuSagecode2011} or retrieved from the LMFDB database \cite{lmfdb} and the examples in \cite{Pizer1980}.
	
\subsection{The setting}
    	
	Let $f=g=h$ be a classical cuspidal modular form corresponding to the unique normalized newform
	\begin{equation*}
	    f(z)= q - q^2 - q^3 - q^4 + q^5 + q^6 + 3 q^8 + q^9+\cdots \in S_2(\Gamma_0(15))^{new}.
	\end{equation*}
	It corresponds to the class of the elliptic curve $y^2+xy+y=x^3+x^2-10x-10$ (Cremona label 15a1, LMFDB label 15.a5), which represents a model for the modular curve $X_0(15)$. We remark that this is not the choice of elliptic curve one should have in mind for applications to the study of the limit point $(2,1,1)$, since its algebraic rank is 0. This form is ordinary at $p=5$, with Hecke eigenvalue $1$, and it has Steinberg automorphic type at $\ell=3$. 

\subsection{The representatives}	

	We launch Franc and Masdeu's algorithm as 
	\begin{verbatim}
	    sage: X = BruhatTitsQuotient(5,3,1);
	    sage: plot(X.plot_fundom()) #Figure 1;
    sage: plot(X.get_graph())   #Figure 2;
	\end{verbatim}
	we obtain the fundamental domain in Figure \ref{figure1} and the quotient graphs of the Bruhat--Tits in Figure \ref{figure2}. The algorithm computes two representatives of the edges in Figure \ref{figure2}, 
	\begin{verbatim}
	    sage: print([X.get_edge_list()[0].rep, X.get_edge_list()[1].rep]);
	\end{verbatim}
	obtaining $\mat{1}{0}{0}{1}$ and $\mat{0}{1}{1}{0}$.
	\begin{figure}[ht!]
        \centering
        \begin{tabular*}{\linewidth}{*{2}{>{\centering\arraybackslash}p{\dimexpr0.5\linewidth-2\tabcolsep}}}
            \includegraphics[width=0.75\linewidth]{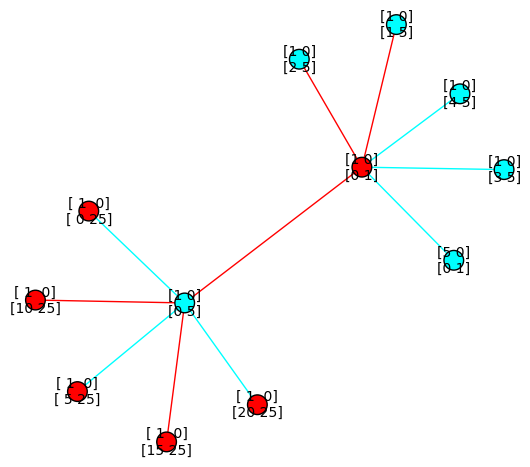}
            &   \includegraphics[width=0.375\linewidth]{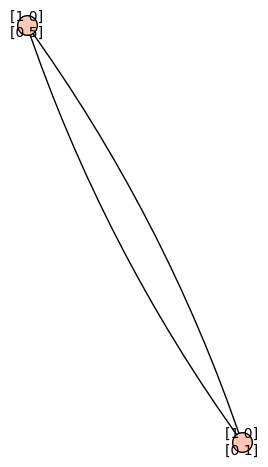}   \\
            \caption{A fundamental domain.}  \label{figure1}
            &   \caption{The quotient graph.}  \label{figure2}
        \end{tabular*}
    \end{figure}
    
    \noindent Therefore, the harmonic cocycle $c_f$ associated with $f$ is the function mapping
	\begin{align*}
	    \Mat{1}{0}{0}{1} \longmapsto 1 && \textrm{ and } && \Mat{0}{1}{1}{0} \longmapsto -1,
	\end{align*}
	with orientation; starting from the center of the fundamental domain, \emph{i.e.} the inner red node in Figure \ref{figure2}, we assign value $+1$ to the red edges with orientation compatible with the edges leaving the center, and value $-1$ to the light blue ones.

\subsection{The quaternionic modular forms}
	
	Let $B$ be the usual quaternion algebra with $\ell=3$ and $R^1$ be an Eichler order of level $5$ in $B$. We can write $B$ with respect to the basis 1, $i$, $j$, $k$, where $i^2=-1$, $j^2=-3$ and $ij=k$. We compute (see Example \ref{example element of order 1/p}) a basis for the Eichler order as
	\begin{equation*}
	    R^1=\langle \tfrac{1}{2}(1+j+2k), \tfrac{1}{2}(i+5k), \tfrac{1}{2}(j+2k),5k\rangle_\z.
	\end{equation*}
	Its class number is 2. As explained in Section \ref{brandt matrices and explicit quaternionic modular forms} and worked out in Example 2 of \cite{Pizer1980}, one computes (enough) Brandt matrices, a few of which are
	\begin{align*}
	    B(2,R^1)=\Mat{1}{2}{2}{1}, && B(3,R^1)=\Mat{0}{1}{1}{0}, && B(5,R^1)=\Mat{6}{5}{5}{6}.
	\end{align*}
	It is readily computed that they are simultaneously diagonalized by
	\begin{equation*}
	    C(R^1)=\Mat{1}{1}{1}{-1}.
	\end{equation*}
	Let $\varphi$ be the quaternion modular form of weight $2$ associated with the vector $\left(\begin{smallmatrix}
		1\\
		-1
	\end{smallmatrix}\right)$; it maps the ideal classes $[I_1=R^1]\mapsto 1$ and $[I_2]\mapsto -1$. Here the ideal $I_2$ is determined by
	\begin{equation*}
	    I_2=\langle 3+j+2k, 3i+5k, 2j+4k, 10k\rangle_\z.
	\end{equation*}
	This form corresponds via the Jacquet--Langlands correspondence to the modular form $f$ considered above. It does not identify immediately with the harmonic cocycle $c_f$, as in Example \ref{example element of order 1/p} we compute that there is a discrepancy between $\left(R^1[1/5]\right)^1$ and $\left(R^1[1/5]\right)^\times$. In particular, the quaternionic modular form identifies with a function $\widetilde{F^1}$ on the graph without orientation, while the harmonic cocycle is a function on the oriented one (see Remark \ref{BD and D remark on norm p}). 
	
\subsection{The function on the edges}\label{function on the edges}

	Even though it is trivial in this example, we explain how to match ideal classes and representatives of the edges, in order to define the value of the form $F^1$ on $Geod_1(\mathcal{T})$ associated with $\varphi$. With respect to the basis of the maximal order
	\begin{equation*}
	    R^1_{3}=\langle \tfrac{1}{2}(1+j), \tfrac{1}{2}(i+k), j, k\rangle_\z,
	\end{equation*}
	one computes the embedding of $B$ into the matrix algebra $M_2(\q_5)$. In particular, one has
	\begin{verbatim}
    sage: u1 = Matrix(QQ,4,1,[2,0,-1,0]);
    sage: M1 = X.embed_quaternion(u1,prec=20)
	\end{verbatim}
	which outputs the matrix
	\begin{equation*}
	    \Mat{1+O(5^{20})}{O(5^{20})}{O(5^{20})}{1+O(5^{20})}.
	\end{equation*}
	Similarly, we consider
	\begin{verbatim}
    sage: u2 = Matrix(QQ,4,1,[0,6,0,-2])
    sage: M2 = X.embed_quaternion(u2,prec=20)
	\end{verbatim}
	with output
	\begin{equation*}
	    \Mat{O(5^{20})}{1 + 4 \cdot 5 + 5^{2} + 4 \cdot 5^{3} + 4 \cdot 5^{4} + \cdots + 2 \cdot 5^{19} + O(5^{20})}{3 + 4 \cdot 5^{2} + 3 \cdot 5^{4} + \cdots + 5^{19}+O(5^{20})}{O(5^{20})},
	\end{equation*}
	where the product of the element on the anti-diagonal can be computed to be $-12$. We retrieve the representative of the class defined by the above matrix, in the quotient of the Bruhat--Tits tree, as
	\begin{verbatim}
	    sage: T = BruhatTitsTree(5);
    sage: print(T.edge(M2));
	\end{verbatim}
	which returns $\mat{0}{1}{1}{0}$. On the other hand, we can recover the starting quaternionic element from
	\begin{verbatim}
	    sage: M_edge = X.get_edge_list()[1].rep;
	    sage: print(X._are_equivalent(M_edge, M_edge, as_edges=True));
	\end{verbatim}
	obtaining
    \begin{equation*}
        \left(\left(\begin{matrix}
            0 & -2 & 0 & 1
        \end{matrix}\right)^t, 0\right).
    \end{equation*}
    We can then test if the vector (in the coordinates of the basis of the maximal order) belongs to one of the ideal classes.

\subsection{The value of the $p$-adic $L$-function}

    In this situation, it is rather easy to determine the cardinality of stabilizers of the edges from the Brandt matrix $B(0,R^1)=\tfrac{1}{4}\mat{1}{1}{1}{1}$. That is, each stabilizer has cardinality $\#Stab_{\left({R^1}\left[1/5\right]\right)^\times}(e)=2$. More generally, one can compute the stabilizers in $(R^1[1/5])^1$ with 
    \phantom{000..}\verb|sage: s = X.get_edge_stabilizers();|\newline
    The length of this vector encodes the cardinality of the stabilizers since we know that the norm-$5$ element is not fixing the edges. Taking into account all the computed data, Proposition \ref{Theta element} reads
    \begin{equation*}
		\Theta_{{(f_\infty^{\times 3})^{B}}'}(\underline{2})= 5^{2}(1-5^{-1})\cdot\sum_{e\in \left({R^1}\left[1/5\right]\right)^\times\backslash \mathcal{E}(\mathcal{T})} \frac{F^1(e)}{2}  F^1(e\left(\begin{smallmatrix}
			1 & 5^{-1}\\
			0 & 1
		\end{smallmatrix}\right))\cdot F^1(e\left(\begin{smallmatrix}
			0 & 5^{-1}\\
			-5 & 0
		\end{smallmatrix}\right)).
	\end{equation*}
	\begin{remark}
	    We recall that in the general case, one would need to apply the Algorithmic routines \ref{tildeF}, \ref{red pn} and \ref{red N}, as well as compute the sets $Geod_n(\mathcal{T})(e)$ and the twist for $g$, as in Remark \ref{geod and twist}.
	\end{remark}
	We check now the action of the two matrices. We multiply by 5 the twisting matrices and take the product with each generator. We then proceed as in Section \ref{function on the edges}, checking the equivalence as edges, and determine to which class they belong. We obtain that
	\begin{equation*}
	    \Mat{1}{0}{0}{1}\cdot \left(\begin{matrix}
			1 & 5^{-1}\\
			0 & 1
		\end{matrix}\right)\sim \Mat{1}{0}{0}{1}\cdot \left(\begin{matrix}
			0 & 5^{-1}\\
			-5 & 0
		\end{matrix}\right) \sim \Mat{0}{1}{1}{0},
	\end{equation*}
	and that
	\begin{align*}
		\Mat{0}{1}{1}{0}\cdot \left(\begin{matrix}
			1 & 5^{-1}\\
			0 & 1
		\end{matrix}\right)\sim \Mat{0}{1}{1}{0}, && \Mat{0}{1}{1}{0}\cdot \left(\begin{matrix}
			0 & 5^{-1}\\
			-5 & 0
		\end{matrix}\right) \sim \Mat{1}{0}{0}{1}.
	\end{align*}
	We finally have
	\begin{equation*}
		\Theta_{{(f_\infty^{\times 3})^{B}}'}(\underline{2})= 5\cdot 2 \cdot \left(1 \cdot (-1)\cdot (-1) + (-1) \cdot (-1)\cdot 1 \right)= 20.
	\end{equation*}
{\setstretch{1.0}
	\bibliography{Bibliography}
	\bibliographystyle{amsalpha}}
{   \hypersetup{hidelinks}	
	\Addresses}

\end{document}